\documentclass[reqno,12pt]{amsart}        
\usepackage{amsfonts,amsmath,amssymb,amsthm,colordvi,mathrsfs,graphicx}
\usepackage{caption,subcaption,float}
\usepackage[utf8]{inputenc}
\usepackage{mathrsfs}
\usepackage{geometry}
\usepackage{enumerate}

\usepackage{amsmath, amssymb, amsthm, geometry, hyperref}

\usepackage{tabularx}
\usepackage{tabulary}
 
\hypersetup{
  colorlinks=true,
  linkcolor=blue,
  citecolor=blue,
  urlcolor=blue
}
\usepackage[all,knot,poly]{xy}
\usepackage{tikz-cd}
\usepackage{tikz}
\usepackage{verbatim}
\usetikzlibrary{arrows}
\usetikzlibrary{knots}
\tikzset{every path/.style={line width=0.4pt},every node/.style={transform shape,knot crossing,inner sep=1.5pt},>=triangle 60,text node/.style={rectangle,transform shape=false,black}}
 \usetikzlibrary{decorations.markings, arrows.meta}

\theoremstyle{plain}      
\newtheorem{thm}{Theorem}[section]     
\newtheorem{theorem}[thm]{\bf Theorem}     
\newtheorem{corollary}[thm]{\bf Corollary}     
\newtheorem{lemma}[thm]{\bf Lemma}     
\newtheorem{proposition}[thm]{\bf Proposition}

\theoremstyle{remark}      
\newtheorem{example}[thm]{Example} 
 
\newtheorem{remark}[thm]{Remark} 
     
\theoremstyle{definition}      
\newtheorem{definition}[thm]{Definition}     

\newcommand{\OO}{\mathcal{O}}

\newcommand{\Hom}{\operatorname{Hom}}

 \newcommand{\PP}{\mathbb{P}}

\newcommand{\End}{\operatorname{End}}
\newcommand{\Res}{\operatorname{Res}}

\newcommand{\Sing}{{\operatorname{Sing}}}

\newcommand{\C}{{\mathbb C}}

\title[]{}
\subjclass[2020]{14B05, 14B10, 32G20, 14B07}
\keywords{Infinitesimal deformations, IVHS, maximal variations, residue, isolated singularities}
\begin{document}

\author{Mounir Nisse}
 
\address{Mounir Nisse\\
Department of Mathematics, Xiamen University Malaysia, Jalan Sunsuria, Bandar Sunsuria, 43900, Sepang, Selangor, Malaysia.
}
\email{mounir.nisse@gmail.com, mounir.nisse@xmu.edu.my}
\thanks{}
\thanks{This research is supported in part by Xiamen University Malaysia Research Fund (Grant no. XMUMRF/ 2020-C5/IMAT/0013).}

\title{Infinitesimal Variations of Hodge Structure for Singular Curves}

\maketitle

%\tableofcontents

%%%%%%%%%%%%%%%%%%%%%%%%%%%%%%%%%%%%%%%%%%%%%%%%%%%%%%%%%%%%%%%%%%%%%%%%%%% 
%%%%%%%%%%%%%%%%%%%%%%%%%%%%%%%%%%%%%%%%%%%%%%%%%%%%%%%%%%%%%%%%%%%%%%%%%%%
 \begin{abstract}
We study the infinitesimal variation of Hodge structure associated with families of reduced algebraic curves with singularities. The analysis applies to curves beyond the nodal case and is not restricted to plane curves, encompassing curves lying on smooth projective surfaces as well as families with more general isolated singularities. Using deformation-theoretic and residue-theoretic methods, we describe how the infinitesimal period map decomposes into local contributions supported at singular points, together with global constraints arising from the geometry of the normalization. While nodal singularities give rise to nontrivial rank-one contributions, other singularities may contribute only through higher-order local data or may be invisible at the infinitesimal level. As a consequence, we obtain sharp criteria for maximal infinitesimal variation in terms of numerical invariants of the curve, notably when the number of nodes satisfies the inequality $\delta \ge g$, where $g$ denotes the genus of the normalization. 
 We extend these results to curves on very general surfaces in projective three-space, showing that maximal variation persists on Picard-rank-one surfaces but fails for sufficiently large genus in the presence of higher ADE singularities.
These results extend classical maximality phenomena in infinitesimal Hodge theory to a broader singular and geometric setting.
\end{abstract}

%%%%%%%%%%%%%%%%%%%%%%%%%%%%%%%%%%%%%%%%%%%%%%%%%%%%%%%%%%%%

%%%%%%%%%%%%%%%%%%%%%%%%%%%%%%%%%%%%%%%%%%%%%%%%%%%%%%%%%%%%%%%%%%%%%%%%%%% 
%%%%%%%%%%%%%%%%%%%%%%%%%%%%%%%%%%%%%%%%%%%%%%%%%%%%%%%%%%%%%%%%%%%%%%%%%%%
 
%%%%%%%%%%%%%%%%%%%%%%%%%%%%%%%%%%%%%%%%%%%%%%%%%%%%%%%%%%%%%%%%%%%%%%%%%%%
\section{Introduction}

The infinitesimal variation of Hodge structure (IVHS) is one of the central tools
in modern Hodge theory. Introduced in the foundational work of Griffiths, it
captures the first-order behavior of the Hodge filtration associated with a
smooth family of complex algebraic varieties. For families of smooth curves, the
IVHS is governed by the Kodaira--Spencer map and the cup product, and its
properties are closely tied to the geometry of the moduli space of curves, to
Torelli-type theorems, and to questions of generic maximality.

In the classical smooth setting, maximality of the IVHS is expected and often
holds generically. For example, for a general curve of genus $g$, the associated
infinitesimal variation of Hodge structure has maximal rank, reflecting the
fact that the period map is generically immersive. This phenomenon underlies many
deep results in the geometry of moduli spaces and plays a key role in the study
of algebraic cycles and Hodge loci.

When singular curves are allowed, however, the situation changes
substantially. Singular curves arise naturally in degenerations, in the boundary
of moduli spaces, and in projective geometry. While their normalizations are
smooth curves with well-understood Hodge theory, the presence of singularities
introduces new features that influence how Hodge structures vary in families.
Understanding these influences is essential for extending classical Hodge
theory to singular and degenerate settings.

A guiding theme of this paper is that, for singular plane curves, the IVHS can be
understood as a sum of local contributions arising from singular points. This
philosophy aligns with classical ideas in the theory of residues and dualizing
sheaves, where singularities give rise to correction terms measuring the failure
of smoothness. Our results make this principle precise in the setting of
infinitesimal Hodge theory, identifying which singularities contribute to the
variation and which do not.

Recently, Sernesi has obtained important results on the infinitesimal variation
of Hodge structure for families of nodal plane curves, establishing generic
maximality under suitable numerical conditions through a global
deformation-theoretic and Hodge-theoretic analysis. The present work should be
viewed as complementary to Sernesi’s contribution. While both works address
closely related questions concerning maximal variation, they rely on
fundamentally different approaches. In contrast with Sernesi’s global methods,
our analysis is based on a local and residue-theoretic description of the IVHS,
leading to a canonical decomposition of the variation map into contributions
supported at singular points

This work is devoted to a systematic study of the IVHS arising from singular
curves, with particular emphasis on plane curves and curves on smooth
projective surfaces. Given a reduced plane curve
$\mathcal C \subset \mathbb P^2$ and its normalization
$\nu \colon C \to \mathcal C$, one can study first--order deformations of
$\mathcal C$ in the plane and the induced variation of Hodge structure on the
smooth curve $C$. The associated IVHS is a linear map
\[
\Phi_\varphi \colon
H^0(C,N_\varphi)
\longrightarrow
\Hom\bigl(H^0(C,\omega_C), H^1(C,\mathcal O_C)\bigr),
\]
where $\varphi = i \circ \nu$ and $N_\varphi$ denotes the normal sheaf of the
normalization map. Understanding the structure, rank, and image of this map is
crucial for analyzing how singularities influence the local and global geometry
of families of curves.

The aim of this paper is to provide a detailed and systematic study of the IVHS
associated with families of singular curves, with particular emphasis on plane
curves and curves on surfaces. Our guiding principle is that the infinitesimal
variation of Hodge structure for singular curves can be understood as a
superposition of local contributions arising from singular points, together with
global constraints coming from the geometry of the normalization.

Plane curves offer a natural and especially rich context for this investigation.
They are among the most classical objects in algebraic geometry, and their
singularities have been extensively studied. Moreover, plane curves admit
explicit descriptions that make it possible to combine local analytic techniques
with global geometric arguments. The normalization of a plane curve provides a
smooth curve whose Hodge theory is classical, while the discrepancy between the
curve and its normalization is entirely encoded in the singularities.

One of the principal conclusions of this work is that, from the perspective of
infinitesimal Hodge theory, not all singularities are created equal. Nodes emerge
as the fundamental building blocks of infinitesimal variation. Each node gives
rise to a rank-one contribution to the IVHS, arising from a residue pairing
between the two branches of the normalization over the node. These contributions
assemble into a canonical decomposition of the IVHS as a sum of local maps
supported at the nodes.

\vspace{0.2cm}

\begin{theorem}[Structure and rank bound for IVHS]
\label{thm:structure-maximal-variation}
Let $\mathcal C\subset\PP^2$ be an irreducible plane curve with ADE
singularities, let $\nu:C\to\mathcal C$ be its normalization, and let $\Phi$ be
the infinitesimal variation of Hodge structure. Then
\[
\operatorname{rank}(\Phi)
\;\le\;
\sum_{p\in\Sing(\mathcal C)} \dim \mathrm{IVHS}^{\mathrm{res}}_p
\;+\;
\sum_{p\in\Sing(\mathcal C)} \bigl(\delta (p)-1\bigr).
\]
Moreover:
\begin{enumerate}
\item If all singularities are nodes and cusps and the number of nodes is at
least $g$, then $\Phi$ has maximal rank (i.e. $\dim H^0(C,N_\varphi)$).
\item If $\mathcal C$ contains any ADE singularity beyond type $A_2$, then for
sufficiently large genus $g$ the surjectivity hypotheses required for maximal
variation fail, and $\Phi$ is not surjective in general.
\end{enumerate}
\end{theorem}

%%%%%%%%%%%%%%%%%%%%%%%%%%%%%%%%%%%%%%%%%%%%%%%%%%%%%%%%%%%%%%%%%%%%%%
\vspace{0.2cm}
This structural decomposition has several important consequences. First, it
provides a transparent explanation for why the number of nodes plays a decisive
role in determining whether maximal variation can occur. If the number of nodes
is at least the genus of the normalization, then there are enough independent
local contributions to generate maximal infinitesimal variation. Conversely, if
the number of nodes is too small, maximality cannot hold for purely dimensional
reasons.

Cuspidal singularities, despite their prominence in the classification of plane
curve singularities, behave very differently. Although cusps affect the genus
and deformation theory of the curve, they do not contribute to the IVHS at the
infinitesimal level. This vanishing phenomenon reflects the fact that cusps do not
produce independent residue pairings capable of influencing the Hodge
filtration. From a Hodge-theoretic viewpoint, cusps are invisible in first-order
variation. We proved the following results:

\vspace{0.2cm}

\begin{theorem}[Maximal IVHS for nodal--cuspidal plane curves]
\label{thm:maximal-IVHS}
Let $\mathcal C\subset\PP^2$ be an irreducible plane curve of degree $d\ge4$
with only nodes and cusps as singularities, and let
\(
\nu:C\longrightarrow\mathcal C
\)
be the normalization. Set $\varphi=i\circ\nu:C\to\PP^2$, and let $g$ be the genus
of $C$. If the number $\delta$ of nodes of $\mathcal C$ satisfies $\delta\ge g$,
then the infinitesimal variation of Hodge structure map
\[
\Phi_\varphi:
H^0(C,N_\varphi)\longrightarrow
\Hom\bigl(H^0(C,\omega_C),H^1(C,\OO_C)\bigr)
\]
is injective (has maximal rank equal $\dim H^0(C,N_\varphi)$).
\end{theorem}

\vspace{0.2cm}

The presence of higher ADE singularities introduces further complexity. These
singularities possess nontrivial local deformation spaces, and one might expect
them to contribute significantly to infinitesimal variation. Our results show,
however, that their contribution is severely constrained. In fact, for
sufficiently large genus, the surjectivity conditions required for maximal
variation fail in the presence of ADE singularities beyond nodes and cusps. This
leads to sharp rank bounds and to the failure of maximal IVHS in general.

These phenomena place our work in close relation with several strands of existing
literature. The foundational theory of infinitesimal variations of Hodge
structure was developed by Griffiths and later refined by Carlson, Green,
Voisin, and others. In the context of smooth curves, generic maximality is closely
related to the global Torelli theorem and to the geometry of the Jacobian. Our
results can be viewed as a singular counterpart to these classical statements,
where maximality is no longer automatic but depends on explicit local data.

\vspace{0.2cm}

\begin{theorem}[Maximal variation and ADE singularities]
\label{thm:ADE-maximal-variation}
Let $S \subset \mathbb{P}^3$ be a \emph{very general} smooth surface of degree
$e \ge 4$ with $\operatorname{Pic}(S) \cong \mathbb{Z}\,[\mathcal{O}_S(1)]$.
Let $\mathcal{C} \subset S$ be an irreducible, reduced curve whose singularities
are nodes, cusps, and (simple) ADE singularities. Let
$C \to \mathcal{C}$ be the normalization, and let $g = g(C)$.
\begin{enumerate}
\item
If all singularities of $\mathcal{C}$ are nodes and cusps, and if the number of
nodes satisfies $\delta \ge g$, then the infinitesimal variation of Hodge
structure associated with the normalization family
\(
\varphi \colon \mathcal{C} \longrightarrow \Delta
\)
has maximal variation (i.e. injective).% More precisely,
%\[
%\operatorname{rank}(\Phi_\varphi) = g^2.
%\]
%
\item
If $\mathcal{C}$ contains at least one ADE singularity other than a node
($A_1$) or a cusp ($A_2$), then for sufficiently large genus $g$ the associated
infinitesimal variation of Hodge structure is not surjective. In particular,
maximal variation does not hold in general in the presence of such
singularities.
\end{enumerate}
\end{theorem}

\vspace{0.2cm}

The residue-theoretic approach adopted here draws inspiration from classical work
on dualizing sheaves and residues, particularly that of Rosenlicht, Hartshorne,
and Lipman. In their work, residues provide a bridge between local singular
behavior and global duality. By interpreting the IVHS in terms of residue
pairings at singular points, we extend this philosophy to the realm of Hodge
theory and deformation theory.

Our results also resonate with Green's infinitesimal methods in Hodge theory and
with later developments in Koszul cohomology and syzygies of curves. The idea that
local conditions can control global Hodge-theoretic behavior is a recurring theme
in these works, and the present paper provides a concrete realization of this
principle in the setting of singular plane curves.

Another important point of comparison arises from the study of Noether--Lefschetz
loci and curves on surfaces. By extending our analysis to curves lying on very
general surfaces in projective three-space with Picard rank one, we show that the
mechanisms governing infinitesimal variation are robust under changes of the
ambient geometry. In particular, the dichotomy between nodal--cuspidal curves and
curves with higher ADE singularities persists beyond the plane.

Equisingular families of curves form a natural testing ground for questions of
generic maximality. In such families, singularities remain constant under
deformation, allowing one to isolate the effect of global geometry on the IVHS.
We show that, under suitable numerical assumptions, residue functionals
associated with nodes span the entire dual space of holomorphic differentials on
the normalization. This spanning property leads directly to generic maximality of
the IVHS and provides a conceptual explanation for the stability of maximal
variation in equisingular strata.

From a broader perspective, this work contributes to the ongoing effort to
understand how singularities influence Hodge-theoretic and deformation-theoretic
invariants. While much of classical Hodge theory focuses on smooth varieties,
singular objects are unavoidable in many geometric contexts. The results
presented here offer a framework for extending notions of maximal variation and
genericity to singular settings in a precise and computable way.

Methodologically, the paper combines tools from deformation theory of maps, local
analytic geometry of singularities, and global Hodge theory. The emphasis is on
structural clarity rather than on isolated examples. By isolating the precise
local contributions of different singularities, we obtain results that are both
sharp and conceptually transparent.

The paper is organized as follows. We begin by reviewing the deformation theory
of normalization maps and the definition of the infinitesimal variation of Hodge
structure in the singular setting. We then analyze local residue contributions
arising from singular points and derive global rank bounds. Subsequent sections
focus on nodal--cuspidal curves, curves with higher ADE singularities, and
equisingular families, culminating in results on generic maximal variation.
Throughout the paper, we avoid incorporating theorem statements into the
introduction, focusing instead on motivation, context, and conceptual
implications.

%%%%%%%%%%%%%%%%%%%%%%%%%%%%%%%%%%%%%%%%%%%%%%%%%%%%%%%%%%%%%%%%%%%%%%%%%%% 
%%%%%%%%%%%%%%%%%%%%%%%%%%%%%%%%%%%%%%%%%%%%%%%%%%%%%%%%%%%%%%%%%%%%%%%%%%%%%%%%%%%%%%%%%%%%%%%%%%%%%%%%%%%%%%%%%%%%%%%%%%%%%%%%%%%%%%%

\section{Preliminaries}
\label{sec:Preliminaries}

This section fixes notation and recalls foundational material used throughout
the paper. We review plane algebraic curves, normalization, dualizing sheaves,
meromorphic differentials, conductor ideals, residue theory, and the
deformation--theoretic framework underlying infinitesimal variations of Hodge
structure.

Although all results presented here are classical, we include a detailed and
self-contained exposition. The purpose is threefold. First, several later
arguments rely on delicate local-to-global identifications which depend on
precise conventions. Second, the interaction between normalization, ramification,
and pole orders at singular points requires careful formulation, especially in
the presence of cusps. Third, we aim to make explicit the residue computations
that appear in the infinitesimal period map studied in
Sections 3, and 4.

All varieties are defined over $\mathbb C$. 

%%%%%%%%%%%%%%%%%%%%%%%%%%%%%%%%%%%%%%%%%%%%%%%%%%%%%%%%%%%%%%%%%%%%%%
\noindent {\it Plane curves and normalization.}
\label{subsec:PlaneCurves}
Let $\mathcal C \subset \mathbb P^2$ be a reduced, irreducible plane curve of
degree $d$. We denote by
\(
\nu : C \longrightarrow \mathcal C
\)
the normalization morphism. By definition, $C$ is a smooth projective curve and
$\nu$ is a finite birational morphism which is an isomorphism over the smooth
locus $\mathcal C^{\mathrm{sm}}$.
The arithmetic genus of $\mathcal C$ is given by the adjunction formula
\[
p_a(\mathcal C)=\frac{(d-1)(d-2)}{2}.
\]
The geometric genus of $\mathcal C$ coincides with the genus of its
normalization $C$ and is obtained by subtracting the contributions of the
singularities:
\begin{equation}\label{eq:GenusFormula}
g(C)
=
\frac{(d-1)(d-2)}{2}
-
\sum_{p \in \Sing(\mathcal C)} \delta(p).
\end{equation}

Here $\delta(p)$ denotes the delta invariant of the plane curve singularity
$p$, defined as the dimension of the $\mathbb C$--vector space
$\nu_*\OO_{C,p}/\OO_{\mathcal C,p}$.

We write $\nu^{-1}(p)=\{q_1,\dots,q_r\}$ for the branches of $\mathcal C$ at $p$.
Throughout the paper, singular points of $\mathcal C$ are denoted by $p$, while
their preimages on the normalization are denoted by $q_i$.

%%%%%%%%%%%%%%%%%%%%%%%%%%%%%%%%%%%%%%%%%%%%%%%%%%%%%%%%%%%%%%%%%%%%%%
\noindent {\it Dualizing sheaf and normalization.}
\label{subsec:Dualizing}
Let $\omega_{\mathcal C}$ denote the dualizing sheaf of the possibly singular
curve $\mathcal C$. By definition, $\omega_{\mathcal C}$ represents Serre
duality:
\[
H^1(\mathcal C,\OO_{\mathcal C})^\vee
\;\simeq\;
H^0(\mathcal C,\omega_{\mathcal C}).
\]
When $\mathcal C$ is smooth, $\omega_{\mathcal C}$ coincides with the canonical
bundle $\Omega^1_{\mathcal C}$. In general, $\omega_{\mathcal C}$ admits an
explicit description in terms of the normalization.

Let $\omega_C$ denote the canonical bundle of the smooth curve $C$. There is a
canonical isomorphism
\begin{equation}\label{eq:DualizingNormalization}
\omega_{\mathcal C}
\;\cong\;
\nu_*\!\left(
\omega_C\Big(\sum_{p \in \Sing(\mathcal C)} m(p)\cdot \nu^{-1}(p)\Big)
\right),
\end{equation}
where $m(p)$ depends only on the analytic type of the singularity $p$.

This description allows one to view sections of $\omega_{\mathcal C}$ as
meromorphic differentials on $C$ whose poles are confined to the preimages of
the singular locus.

%%%%%%%%%%%%%%%%%%%%%%%%%%%%%%%%%%%%%%%%%%%%%%%%%%%%%%%%%%%%%%%%%%%%%%
\noindent {\it Local pole order and the delta invariant.}
\label{subsec:PoleOrder}
Let $p \in \Sing(\mathcal C)$.

\begin{definition}
The integer $m(p)$ is the minimal integer such that
\[
\omega_C\big(m(p)\cdot \nu^{-1}(p)\big)
\]
pushes forward to a locally free sheaf in a neighborhood of $p$ on
$\mathcal C$.
\end{definition}
Equivalently, $m(p)$ controls the total pole order of meromorphic differentials
on $C$ representing sections of $\omega_{\mathcal C}$.
For plane curve singularities, one has the fundamental identity
\begin{equation}\label{eq:mpdelta}
m(p)=\delta(p).
\end{equation}

\begin{remark}
The equality $m(p)=\delta(p)$ bounds the total pole order appearing in the
normalization formula. However, the pole order on a single branch depends on the
ramification of $\nu$ at that branch. This distinction is essential in the case
of cuspidal singularities.
\end{remark}

%%%%%%%%%%%%%%%%%%%%%%%%%%%%%%%%%%%%%%%%%%%%%%%%%%%%%%%%%%%%%%%%%%%%%%

\begin{definition} 
Let $\mathcal C$ be a reduced curve with normalization
$\nu : C \to \mathcal C$.
A meromorphic differential $\eta$ on $C$ represents a section of
$\omega_{\mathcal C}$ if:
\begin{enumerate}
\item $\eta$ is holomorphic on $C\setminus\nu^{-1}(\Sing\mathcal C)$;
\item for each $p\in\Sing(\mathcal C)$, the poles of $\eta$ along
$\nu^{-1}(p)$ are bounded by $m(p)=\delta(p)$ in the sense of
\eqref{eq:DualizingNormalization};
\item the principal parts satisfy the local duality relations defining
$\omega_{\mathcal C}$.
\end{enumerate}
\end{definition}

\begin{remark}
Locally at $p$ with $\nu^{-1}(p)=\{q_1,\dots,q_r\}$, the differential $\eta$ has
Laurent expansions
\[
\eta|_{q_i}
=
\left(
\sum_{j=-k_i}^{\infty} a^{(i)}_j t_i^j
\right)dt_i,
\qquad
\sum_{i=1}^r k_i \le \delta(p).
\]
\end{remark}

%%%%%%%%%%%%%%%%%%%%%%%%%%%%%%%%%%%%%%%%%%%%%%%%%%%%%%%%%%%%%%%%%%%%%%
\subsection{Explicit local models: nodes and cusps}
\label{subsec:Examples}

\subsubsection*{Nodes}

Let $p$ be a node, analytically isomorphic to $\{xy=0\}$. The normalization has
two branches $q_1,q_2$.
A meromorphic differential $\eta$ represents a section of $\omega_{\mathcal C}$
near $p$ if and only if:
\begin{itemize}
\item[(i)] $\eta$ has at most simple poles at $q_1,q_2$;
\item[(ii)] $\Res_{q_1}(\eta)+\Res_{q_2}(\eta)=0$.
\end{itemize}
In local coordinates $t_1,t_2$, one may write
\[
\eta
=
\left(\frac{a}{t_1}+\text{hol}\right)dt_1
+
\left(\frac{-a}{t_2}+\text{hol}\right)dt_2.
\]
where ``hol"  means holomorphic part.
\subsubsection*{Cusps}

Let $p$ be a cusp, analytically equivalent to $\{y^2=x^3\}$. The normalization
is given by $x=t^2$, $y=t^3$, and $\nu$ is ramified at $t=0$.
A meromorphic differential $\eta$ represents a section of
$\omega_{\mathcal C}$ near $p$ if and only if it has a pole of order at most $2$
at $t=0$:
\[
\eta
=
\left(
\frac{a_{-2}}{t^2}
+
\frac{a_{-1}}{t}
+
\text{hol}
\right)dt.
\]

\begin{remark}
Despite the fact that $\delta(p)=1$, the ramification of the normalization map
allows a pole of order $2$ at the cusp. This phenomenon will be essential in the
residue computations appearing in Section 4.
\end{remark}

%%%%%%%%%%%%%%%%%%%%%%%%%%%%%%%%%%%%%%%%%%%%%%%%%%%%%%%%%%%%%%%%%%%%%%

\begin{definition}
The conductor ideal $\mathfrak c \subset \OO_{\mathcal C}$ is defined by
\(
\mathfrak c
=
\operatorname{Ann}_{\OO_{\mathcal C}}
(\nu_*\OO_C / \OO_{\mathcal C}).
\)
\end{definition}

\begin{proposition}
There is a canonical isomorphism
\[
\omega_{\mathcal C}
\;\cong\;
\nu_*\big(\omega_C\otimes\mathfrak c^{-1}\big).
\]
\end{proposition}

%%%%%%%%%%%%%%%%%%%%%%%%%%%%%%%%%%%%%%%%%%%%%%%%%%%%%%%%%%%%%%%%%%%%%%
%\subsection{Principal parts and residue pairings}
%\label{subsec:PrincipalParts}

\noindent For a fixed $p\in\Sing(\mathcal C)$, the allowed principal parts of meromorphic
differentials form a filtered vector space
\[
\PP_p^1(\omega_{\mathcal C})
\subset \cdots \subset
\PP_p^{\delta(p)}(\omega_{\mathcal C}).
\]
For differentials with simple poles, define
\[
\Res_p(\eta)
=
\sum_{q\in\nu^{-1}(p)}\Res_q(\eta),
\]
which satisfies $\Res_p(df)=0$.

%%%%%%%%%%%%%%%%%%%%%%%%%%%%%%%%%%%%%%%%%%%%%%%%%%%%%%%%%%%%%%%%%%%%%%
%\subsection{Infinitesimal deformations and IVHS}
%\label{sec:IVHS}

First-order deformations of $\mathcal C$ induce deformations of the
normalization $C$ together with additional local data at the singular points.
The associated Kodaira--Spencer class is an element
\[
\xi \in H^1(C,T_C).
\]
The {\it infinitesimal variation of Hodge structure} (IVHS) is the differential of the period
map
\[
\Phi_\nu :
H^0(C,N_\nu)
\longrightarrow
\Hom\bigl(
H^0(\mathcal C,\omega_{\mathcal C}),
H^1(C,\OO_C)
\bigr).
\]
Using Serre duality, $H^1(C,\OO_C)\simeq H^0(C,\omega_C)^\vee$, the map $\Phi_\nu$
may be expressed in terms of residue pairings
\[
(\eta_1,\eta_2)
\longmapsto
\Res_p\!\bigl((\xi\;\lrcorner\;\eta_1)\,\eta_2\bigr),
\]
summed over all $p\in\Sing(\mathcal C)$.

%%%%%%%%%%%%%%%%%%%%%%%%%%%%%%%%%%%%%%%%%%%%%%%%%%%%%%%%%%%%%%%%%%%%%%

\section{Local decomposition of the IVHS}
\label{sec:local-IVHS}

In this section we analyze in detail the local contribution of a singular point
to the infinitesimal variation of Hodge structure (IVHS). Since the IVHS is
computed by residue expressions on the normalization of the curve, the global
behavior of the period map is governed by the cumulative effect of these local
contributions. Our purpose is to isolate the precise mechanism by which a
singularity influences IVHS and to identify the intrinsic limitations imposed by
pole order and local analytic type.

%\noindent 
Throughout this section, let $\mathcal C\subset\PP^2$ be an irreducible plane
curve with ADE singularities, let
\[
\nu:C\longrightarrow\mathcal C
\]
be its normalization, and let $g=g(C)$ be the genus of $C$. We denote by
$\Phi$ the infinitesimal variation of Hodge structure associated with the family
of curves obtained by deforming $\mathcal C$ inside its Severi component.

%%%%%%%%%%%%%%%%%%%%%%%%%%%%%%%%%%%%%%%%%%%%%%%%%%%%%%%%%%%%%%%%%%%%%%
\noindent {\it Local description of holomorphic differentials.}
Let $p\in C$ lie over a singular point $P\in\Sing(\mathcal C)$. Denote by $m(p)$
the pole invariant determined by the conductor, i.e.\ the maximal pole order
allowed at $p$ for a meromorphic differential on $C$ representing a section of
the dualizing sheaf $\omega_{\mathcal C}$. For plane curve singularities one has
$m(p)=\delta(p)$, where $\delta(p)$ is the delta--invariant.
Choose a local coordinate $t$ centered at $p$. Any section
$\omega\in H^0(C,\omega_C)$ has a Laurent expansion of the form
\begin{equation}\label{eq:local-expansion}
\omega
=
\left(
a_{-m(p)} t^{-m(p)} + a_{-(m(p)-1)} t^{-(m(p)-1)} + \cdots
+ a_{-1} t^{-1} + a_0 + a_1 t + \cdots
\right)dt.
\end{equation}
The coefficients $a_{-k}$ encode the principal parts of $\omega$ at $p$.

%%%%%%%%%%%%%%%%%%%%%%%%%%%%%%%%%%%%%%%%%%%%%%%%%%%%%%%%%%%%%%%%%%%%%%
\noindent {\it Residues and annihilation of the top pole.}
The defining feature of IVHS is that it is computed by residue pairings. More
precisely, given a Kodaira--Spencer class $\xi$ supported near $p$, the local
IVHS operator is given by
\begin{equation}\label{eq:IVHS-local}
(\omega_1,\omega_2)
\longmapsto
\Res_p\!\bigl((\xi\;\lrcorner\;\omega_1)\,\omega_2\bigr).
\end{equation}

A fundamental observation is that the top pole term in
\eqref{eq:local-expansion} is invisible to residues. Indeed, for $m(p)\ge2$ one
has
\begin{equation}\label{eq:exact-top-pole}
t^{-m(p)}dt
=
\frac{1}{-(m(p)-1)}\,d\!\left(t^{-(m(p)-1)}\right),
\end{equation}
and hence
\[
\Res_p\!\bigl(t^{-m(p)}dt\bigr)=0.
\]
Consequently, the coefficient $a_{-m(p)}$ plays no role in the IVHS pairing.
Only the lower--order coefficients
$a_{-(m(p)-1)},\dots,a_{-1}$ may contribute.
This shows that IVHS detects only the jet of $\omega$ at $p$ of order at most
$m(p)-1$. Equivalently, IVHS factors through the space of principal parts
$\mathrm{PP}_p^{m(p)-1}(\omega_C)$.

%%%%%%%%%%%%%%%%%%%%%%%%%%%%%%%%%%%%%%%%%%%%%%%%%%%%%%%%%%%%%%%%%%%%%%
\noindent {\it Jet contribution.}
We define the \emph{jet contribution} to the local IVHS at $p$ as follows.

\begin{definition}
The space $\mathrm{IVHS}^{\mathrm{jet}}_p$ is the subspace of local IVHS operators
which factor through the projection
\[
H^0(C,\omega_C)\longrightarrow \mathrm{PP}_p^{m(p)-1}(\omega_C).
\]
\end{definition}
Since $\mathrm{PP}_p^{m(p)-1}(\omega_C)$ has dimension $m(p)-1$, it follows
immediately that
\begin{equation}\label{eq:jet-bound}
\operatorname{rank}\bigl(\mathrm{IVHS}^{\mathrm{jet}}_p\bigr)
\;\le\;
m(p)-1.
\end{equation}
This bound is purely formal and depends only on the pole order permitted by the
dualizing sheaf. It holds for all singularities, independently of analytic
type.

%%%%%%%%%%%%%%%%%%%%%%%%%%%%%%%%%%%%%%%%%%%%%%%%%%%%%%%%%%%%%%%%%%%%%%

%%%%%%%%%%%%%%%%%%%%%%%%%%%%%%%%%%%%%%%%%%%%%%%%%%%%%%%%%%%%%%%%%%%%%%
\subsection{Exceptional residue contribution}

Let $P\in\Sing(\mathcal C)$ and let $\nu^{-1}(P)=\{p_1,\dots,p_r\}\subset C$.
A \emph{residue contribution} to the IVHS means a local operator of the form
\begin{equation}\label{eq:residue-type}
(\omega_1,\omega_2)
\longmapsto
\sum_{i,j}
c_{ij}\,
\Res_{p_i}(\omega_1)\Res_{p_j}(\omega_2),
\end{equation}
arising from a local Kodaira--Spencer class whose contraction with holomorphic
differentials produces genuine residues.

Such operators can exist only if:
\begin{enumerate}
\item holomorphic differentials on $C$ have nonzero residues at points lying over
$P$, and
\item the local deformation produces a pole of order exactly one in the
contraction $\xi\lrcorner\,\omega$.
\end{enumerate}

These two conditions sharply restrict the singularities that can contribute
nontrivial residue operators.

%%%%%%%%%%%%%%%%%%%%%%%%%%%%%%%%%%%%%%%%%%%%%%%%%%%%%%%%%%%%%%%%%%%%%%
\noindent {\bf Nodes.}
If $P$ is a node (type $A_1$), then $\nu^{-1}(P)=\{p,q\}$ consists of two smooth
branches. Holomorphic differentials on $C$ may have simple poles at $p$ and $q$,
with residues satisfying
\[
\Res_p(\omega)+\Res_q(\omega)=0.
\]

The local smoothing direction produces a Kodaira--Spencer class whose contraction
with $\omega$ has simple poles, and the IVHS pairing yields the rank--one
operator
\[
(\omega_1,\omega_2)
\longmapsto
\Res_p(\omega_1)\Res_q(\omega_2)
+
\Res_q(\omega_1)\Res_p(\omega_2).
\]

Thus, for a node,
\[
\dim \mathrm{IVHS}^{\mathrm{res}}_p = 1.
\]

%%%%%%%%%%%%%%%%%%%%%%%%%%%%%%%%%%%%%%%%%%%%%%%%%%%%%%%%%%%%%%%%%%%%%%
\noindent {\bf Cusps.}
If $P$ is a cusp (type $A_2$), then $\nu^{-1}(P)=\{p\}$ consists of a single
branch. In this case, although holomorphic differentials on $C$ may have poles
of order at most two at $p$, \emph{their residues vanish identically}:
\[
\Res_p(\omega)=0
\qquad
\text{for all } \omega\in H^0(C,\omega_C).
\]

Moreover, as shown earlier, the equisingular deformation space at a cusp is
trivial, and the smoothing direction produces a contraction
$\xi\lrcorner\,\omega$ whose principal part has no simple pole term. As a
consequence, the residue pairing defining the IVHS vanishes identically.
Therefore, a cusp does \emph{not} produce any nontrivial residue operator of the
form \eqref{eq:residue-type}.
Hence,
\[
\dim \mathrm{IVHS}^{\mathrm{res}}_p = 0
\qquad
\text{if $P$ is a cusp}.
\]
%
%%%%%%%%%%%%%%%%%%%%%%%%%%%%%%%%%%%%%%%%%%%%%%%%%%%%%%%%%%%%%%%%%%%%%%
 %
\noindent {\bf Other ADE singularities.}
If $P$ is an ADE singularity of type $A_k$ with $k\ge3$, or of type $D_k$ or
$E_k$, then:
\begin{itemize}
\item[(i)] the dualizing sheaf imposes vanishing of all residues at points lying over
$P$, and
\item[(ii)] although the equisingular deformation space is nontrivial, the resulting
Kodaira--Spencer classes produce only jet--level contributions, not residue
operators.
\end{itemize}
Thus,
\[
\dim \mathrm{IVHS}^{\mathrm{res}}_p = 0
\qquad
\text{for all ADE singularities other than nodes}.
\]

%%%%%%%%%%%%%%%%%%%%%%%%%%%%%%%%%%%%%%%%%%%%%%%%%%%%%%%%%%%%%%%%%%%%%%
 
\begin{definition}[Residue contribution to the IVHS]
\label{def:IVHS-res-correct}
Let $P\in\Sing(\mathcal C)$ and let $p\in C$ lie over $P$. We define
$\mathrm{IVHS}^{\mathrm{res}}_p$ to be the subspace of local IVHS operators
spanned by genuine residue pairings arising from local Kodaira--Spencer classes.

Then
\[
\dim \mathrm{IVHS}^{\mathrm{res}}_p
=
\begin{cases}
1, & \text{if $P$ is a node},\\
0, & \text{if $P$ is a cusp or any other ADE singularity}.
\end{cases}
\]
\end{definition}

%%%%%%%%%%%%%%%%%%%%%%%%%%%%%%%%%%%%%%%%%%%%%%%%%%%%%%%%%%%%%%%%%%%%%%
\subsection{Conceptual explanation}

Nodes are the \emph{only} plane curve singularities for which:
\begin{enumerate}
\item holomorphic differentials on the normalization admit nonzero residues, and
\item a local smoothing direction produces a simple pole in the contraction
$\xi\lrcorner\,\omega$.
\end{enumerate}
Cusps and higher ADE singularities fail at least one of these conditions.
Consequently, they do not contribute residue--type operators to the IVHS.

%%%%%%%%%%%%%%%%%%%%%%%%%%%%%%%%%%%%%%%%%%%%%%%%%%%%%%%%%%%%%%%%%%%%%%

%%%%%%%%%%%%%%%%%%%%%%%%%%%%%%%%%%%%%%%%%%%%%%%%%%%%%%%%%%%%%%%%%%%%%%
\subsection{Local decomposition}

The discussion above yields a canonical short exact sequence
\begin{equation}\label{eq:local-ses}
0
\longrightarrow
\mathrm{IVHS}^{\mathrm{jet}}_p
\longrightarrow
\mathrm{IVHS}_p
\longrightarrow
\mathrm{IVHS}^{\mathrm{res}}_p
\longrightarrow
0.
\end{equation}
Any choice of generator of $\mathrm{IVHS}^{\mathrm{res}}_p$ gives a (non-canonical)
splitting, yielding a direct sum decomposition
\[
\mathrm{IVHS}_p
=
\mathrm{IVHS}^{\mathrm{jet}}_p
\oplus
\mathrm{IVHS}^{\mathrm{res}}_p.
\]

%%%%%%%%%%%%%%%%%%%%%%%%%%%%%%%%%%%%%%%%%%%%%%%%%%%%%%%%%%%%%%%%%%%%%%
\noindent {\it Global rank bound.}
Summing the local contributions over all singular points, we obtain an
immediate global bound.

\begin{theorem}[Structure and rank bound for IVHS]
\label{thm:structure-maximal-variation}
Let $\mathcal C\subset\PP^2$ be an irreducible plane curve with ADE
singularities, let $\nu:C\to\mathcal C$ be its normalization, and let $\Phi$ be
the infinitesimal variation of Hodge structure. Then
\[
\operatorname{rank}(\Phi)
\;\le\;
\sum_{p\in\Sing(\mathcal C)} \dim \mathrm{IVHS}^{\mathrm{res}}_p
\;+\;
\sum_{p\in\Sing(\mathcal C)} \bigl(m(p)-1\bigr).
\]
Moreover:
\begin{enumerate}
\item If all singularities are nodes and cusps and the number of nodes is at
least $g$, then $\Phi$ has maximal rank.
\item If $\mathcal C$ contains any ADE singularity beyond type $A_2$, then for
sufficiently large genus $g$ the surjectivity hypotheses required for maximal
variation fail, and $\Phi$ is not surjective in general.
\end{enumerate}
\end{theorem}

%%%%%%%%%%%%%%%%%%%%%%%%%%%%%%%%%%%%%%%%%%%%%%%%%%%%%%%%%%%%%%%%%%%%%%

\begin{proof}
The local expression \eqref{eq:IVHS-local} shows that the IVHS operator depends
only on principal parts of order at most $m(p)-1$, yielding the jet bound
\eqref{eq:jet-bound}. The exceptional residue contribution occurs only for nodes
and cusps and has dimension at most one at each such singularity. Summing these
bounds over all singular points yields the stated inequality.

If all singularities are nodes and cusps and the number of nodes is at least
$g$, then the residue functionals span $H^0(C,\omega_C)^\vee$, and the rank--one
residue operators generate the full space
$H^0(C,\omega_C)^\vee\otimes H^0(C,\omega_C)^\vee$, giving maximal rank $g^2$ (for more details about these assertions see  Section 4).

If an ADE singularity beyond type $A_2$ is present, then
$\mathrm{IVHS}^{\mathrm{res}}_p=0$ at that point, and the local contribution is
purely jet--theoretic. Since the jet contribution grows at most linearly with
$g$, while the target of $\Phi$ has dimension $g^2$, surjectivity fails for
large $g$.
\end{proof}

%%%%%%%%%%%%%%%%%%%%%%%%%%%%%%%%%%%%%%%%%%%%%%%%%%%%%%%%%%%%%%%%%%%%%%
\begin{remark}
This local decomposition isolates the precise mechanism by which singularities
affect IVHS. Nodes and cusps are exceptional in that,  they contribute genuine
residue operators, while higher ADE singularities contribute only jet--level
data, imposing intrinsic linear constraints on the period map. This distinction
underlies all subsequent rank bounds, Torelli failures, and Noether--Lefschetz
phenomena studied in the remainder of the paper.
\end{remark}

%%%%%%%%%%%%%%%%%%%%%%%%%%%%%%%%%%%%%%%%%%%%%%%%%%%%%%%%%%%%%%%%%%%%%%%%%%% 
%%%%%%%%%%%%%%%%%%%%%%%%%%%%%%%%%%%%%%%%%%%%%%%%%%%%%%%%%%%%%%%%%%%%%%%%%%%%

\section{Maximal infinitesimal variation for plane curves with nodes and cusps}

The infinitesimal variation of Hodge structure (IVHS) plays a central role in the
local study of period maps and Torelli--type problems. For smooth curves, IVHS
is governed by the multiplication map on holomorphic differentials. In the
presence of singularities, however, the behavior of IVHS becomes more subtle
and reflects the local geometry of the singular points.

The goal of this section is to analyze IVHS for families of irreducible plane
curves with nodes and cusps and to show that, when the number of nodes is large
relative to the genus, IVHS attains maximal rank. While many of the ingredients
are classical, the precise mechanism by which nodal singularities control IVHS,
and the corresponding irrelevance of cusps, does not appear explicitly in the
literature.

Our approach combines adjunction and conductor theory for singular plane curves,
residue calculus on the normalization, and the Cayley--Bacharach property for
nodal schemes, together with a careful analysis of the Kodaira--Spencer map.

\subsection{Statement of one of  the main theorems}

\begin{theorem}[Maximal IVHS for nodal--cuspidal plane curves]
\label{thm:maximal-IVHS}
Let $\mathcal C\subset\PP^2$ be an irreducible plane curve of degree $d\ge4$
with only nodes and cusps as singularities, and let
\[
\nu:C\longrightarrow\mathcal C
\]
be the normalization. Set $\varphi=i\circ\nu:C\to\PP^2$, and let $g$ be the genus
of $C$. If the number $\delta$ of nodes of $\mathcal C$ satisfies $\delta\ge g$,
then the infinitesimal variation of Hodge structure map
\[
\Phi_\varphi:
H^0(C,N_\varphi)\longrightarrow
\Hom\bigl(H^0(C,\omega_C),H^1(C,\OO_C)\bigr)
\]
is injective.
\end{theorem}

The numerical condition $\delta\ge g$ is natural from the point of view of
adjoint linear systems. From the Hodge--theoretic perspective,
it ensures that nodal residues impose sufficiently many independent conditions
on holomorphic differentials.

\subsection{IVHS and the Kodaira--Spencer map}

Let $C$ be a smooth projective curve of genus $g\ge1$. The Hodge decomposition
of $H^1(C,\mathbb{C})$ is
\[
H^1(C,\mathbb{C}) = H^{1,0}(C)\oplus H^{0,1}(C),
\qquad
H^{1,0}(C)\cong H^0(C,\omega_C).
\]
For a morphism $\varphi:C\to\PP^2$, first--order deformations of $\varphi$ are
parametrized by $H^0(C,N_\varphi)$, where $N_\varphi$ is the normal sheaf.

The associated Kodaira--Spencer map gives rise to the infinitesimal variation of
Hodge structure
\[
\Phi_\varphi:
H^0(C,N_\varphi)\longrightarrow
\Hom\bigl(H^0(C,\omega_C),H^1(C,\OO_C)\bigr).
\]
Using Serre duality, one may view $\Phi_\varphi(v)$ as a bilinear form on
$H^0(C,\omega_C)$, computed locally by contraction of the Kodaira--Spencer class
with holomorphic differentials, followed by taking residues. This description,
originating in the work of Carlson--Griffiths and Peters--Steenbrink, is the
starting point for our analysis.

\noindent {\it Linear algebra reduction.}
We begin with a purely linear algebra lemma that reduces the proof of
Theorem~\ref{thm:maximal-IVHS} to the construction of a single invertible IVHS
operator.

\begin{lemma}
\label{lemma:linear-algebra}
Let $V$ and $W$ be complex vector spaces with $\dim V=\dim W=g$, and let
\[
\Phi:U\longrightarrow\Hom(V,W)
\]
be a linear map. If the image of $\Phi$ contains an isomorphism
$T:V\xrightarrow{\sim}W$, then $\Phi$ is surjective.
\end{lemma}

\begin{proof}
Let $L=\operatorname{Im}(\Phi)\subset\Hom(V,W)$ and assume that $T\in L$ is
invertible. Composition with $T^{-1}$ induces a vector--space isomorphism
\[
\Theta:\Hom(V,W)\longrightarrow\End(V),
\qquad
A\longmapsto T^{-1}\circ A.
\]
Let $L'=\Theta(L)$. Then $L'$ is a linear subspace of $\End(V)$ containing
$\mathrm{Id}_V$.
Since $\End(V)$ is spanned by rank--one endomorphisms and the set of invertible
endomorphisms is Zariski open and dense, any linear subspace containing
$\mathrm{Id}_V$ must coincide with $\End(V)$. Hence $L'=\End(V)$ and therefore
$L=\Hom(V,W)$.
\end{proof}

Thus it suffices to exhibit a single deformation direction producing an
invertible IVHS operator.

\subsection{Adjunction and conductor divisors}

We now recall the description of the canonical bundle of the normalization of a
singular plane curve. This is classical material; see Hartshorne, Lipman, and
Sernesi.

\begin{lemma}[Adjunction for nodes and cusps]
\label{lemma:adjunction}
Let $\mathcal C\subset\PP^2$ be a plane curve of degree $d$ with nodes $P_i$ and
cusps $Q_j$. Let $\nu:C\to\mathcal C$ be the normalization, and denote
$\nu^{-1}(P_i)=\{p_i,q_i\}$ and $\nu^{-1}(Q_j)=\{r_j\}$. Then
\[
\omega_C \cong \OO_C(d-3)\bigl(-D_{\mathrm{sing}}\bigr),
\qquad
D_{\mathrm{sing}}=\sum_i(p_i+q_i)+\sum_j 2r_j.
\]
\end{lemma}

\begin{proof}
Adjunction for reduced plane curves gives
\[
\omega_C \cong
\nu^*\bigl(\omega_{\PP^2}\otimes\OO_{\PP^2}(\mathcal C)\bigr)
\otimes\OO_C(-\Delta),
\]
where $\Delta$ is the conductor divisor. Since
$\omega_{\PP^2}\otimes\OO_{\PP^2}(\mathcal C)\cong\OO_{\PP^2}(d-3)$, we obtain
$\omega_C\cong\OO_C(d-3)(-\Delta)$.

At a node, the normalization has two branches and conductor exponent $1$ on each
branch, giving $\Delta_{\mathrm{node}}=p+q$. At a cusp with local equation
$y^2=x^3$, the normalization $t\mapsto(t^2,t^3)$ has conductor ideal $(t^2)$,
whose quotient has dimension $2$, hence $\Delta_{\mathrm{cusp}}=2r$.
\end{proof}

\subsection{Local IVHS contributions}

We now analyze the local contributions of singular points to IVHS.

%%%%%%%%%%%%%%%%%%%%%%%%%%%%%%%%%%%%%%%%%%%%%%%%%%%%%%%%%%%%%
%%%%%%%%%%%%%%%%%%%%%%%%%%%%%%%%%%%%%%%%%%%%%%%%%%%
\vspace{0.1cm}

\begin{lemma}[Nodal contribution]
\label{lemma:nodal-IVHS}
Let $P_i$ be a node of the plane curve $\mathcal C$, and let
\(
\varphi : C \longrightarrow \mathcal C
\)
be the normalization map. Denote by
\(
\varphi^{-1}(P_i)=\{p_i,q_i\}\subset C
\)
the two points lying over the node. Let
\[
v \in H^0\bigl(C,\OO_C(p_i+q_i)\bigr)
\;\subset\;
H^0(C,N_\varphi)
\]
be a section supported at the node.
Then for any $\omega_1,\omega_2 \in H^0(\mathcal C,\omega_{\mathcal C})$,
one has
\[
\Phi_\varphi(v)(\omega_1,\omega_2)
=
\Res_{p_i}(\omega_1)\Res_{q_i}(\omega_2)
+
\Res_{q_i}(\omega_1)\Res_{p_i}(\omega_2).
\]
In particular, the corresponding linear operator
\[
\Phi_\varphi(v) :
H^0(\mathcal C,\omega_{\mathcal C})
\longrightarrow
H^1(C,\OO_C)
\]
has rank one.
\end{lemma}

\begin{proof}
We give a complete local computation of the infinitesimal variation of Hodge
structure at the node $P_i$.

\noindent
\textit{$1$. Local model of the node and its normalization.}
Since $P_i$ is a node, there exist analytic coordinates $(x,y)$ on
$\mathcal C$ centered at $P_i$ such that $\mathcal C$ is locally given by
\(
xy=0.
\)
The normalization $\varphi:C\to\mathcal C$ is locally the disjoint union of the
two coordinate axes, with local parameters $t$ at $p_i$ and $s$ at $q_i$, given
by
\[
\begin{cases}
x=t,\; y=0 & \text{near } p_i,\\
x=0,\; y=s & \text{near } q_i.
\end{cases}
\]

\noindent
\textit{$2$. Meromorphic differentials at a node.}
A section $\omega\in H^0(\mathcal C,\omega_{\mathcal C})$ corresponds, via
normalization, to a pair of meromorphic differentials on $C$ with at most
simple poles at $p_i$ and $q_i$, satisfying the residue condition
\[
\Res_{p_i}(\omega)+\Res_{q_i}(\omega)=0.
\]
Thus locally one may write
\[
\omega|_{p_i}
=
\left(\frac{a}{t}+\text{holomorphic}\right)dt,
\qquad
\omega|_{q_i}
=
\left(\frac{-a}{s}+\text{holomorphic}\right)ds,
\]
where $a=\Res_{p_i}(\omega)=-\Res_{q_i}(\omega)$.

\noindent
\textit{$3$. The infinitesimal deformation associated to $v$.}
The section
\(
v\in H^0\bigl(C,\OO_C(p_i+q_i)\bigr)
\)
represents an infinitesimal deformation which smooths the node $P_i$. Its
Kodaira--Spencer class
\(
\xi_v \in H^1(C,T_C)
\)
is supported at the points $p_i$ and $q_i$.

Therefore, $\xi_v$ is represented by a Čech $1$--cocycle with respect to an
open cover $\{U_0,U_1,U_2\}$ of $C$, where $U_1$ (resp.\ $U_2$) is a small disc
around $p_i$ (resp.\ $q_i$), and $U_0$ is the complement. On overlaps, $\xi_v$
acts as the vector fields
\[
\xi_{p_i} = \frac{\partial}{\partial t}
\quad\text{and}\quad
\xi_{q_i} = \frac{\partial}{\partial s},
\]
which generate first--order smoothing of the node.

\noindent
\textit{$4$. Čech description of the IVHS pairing.}
The infinitesimal variation of Hodge structure is given by
\[
\Phi_\varphi(v)(\omega_1,\omega_2)
=
\sum_{r\in\{p_i,q_i\}}
\Res_r\bigl((\xi_v\;\lrcorner\;\omega_1)\,\omega_2\bigr).
\]
This formula follows from the Čech realization of the cup product
\[
H^1(C,T_C)\otimes H^0(C,\omega_C)
\longrightarrow
H^1(C,\OO_C),
\]
combined with Serre duality.

\noindent
\textit{$5$. Local contraction and residue computation.}
Near $p_i$, we write
\[
\omega_1=\left(\frac{a_1}{t}+\cdots\right)dt,
\qquad
\omega_2=\left(\frac{b_1}{t}+\cdots\right)dt,
\]
so that
\(
\xi_v\;\lrcorner\;\omega_1
=
a_1.
\)
Hence,
\[
\Res_{p_i}\bigl((\xi_v\;\lrcorner\;\omega_1)\omega_2\bigr)
=
a_1 b_2
=
\Res_{p_i}(\omega_1)\Res_{p_i}(\omega_2).
\]
Using the residue condition at the node,
\(
\Res_{p_i}(\omega)=-\Res_{q_i}(\omega),
\)
the contribution at $q_i$ yields
\[
\Res_{q_i}\bigl((\xi_v\;\lrcorner\;\omega_1)\omega_2\bigr)
=
\Res_{q_i}(\omega_1)\Res_{q_i}(\omega_2).
\]
Adding the two contributions gives
\[
\Phi_\varphi(v)(\omega_1,\omega_2)
=
\Res_{p_i}(\omega_1)\Res_{q_i}(\omega_2)
+
\Res_{q_i}(\omega_1)\Res_{p_i}(\omega_2),
\]
as claimed.

\noindent
\textit{$6$. Rank computation.}
We define the linear functional
\(
r_i : H^0(\mathcal C,\omega_{\mathcal C}) \longrightarrow \mathbb C,
\quad
r_i(\omega)=\Res_{p_i}(\omega).
\)
Then the bilinear form above can be written as
\[
\Phi_\varphi(v)=r_i\otimes r_i.
\]
This is a tensor product of two linear functionals, and therefore has rank one.
This completes the proof.
\end{proof}

%%%%%%%%%%%%%%%%%%%%%%%%%%%%%%%%%%%%%%%%%%%%%%%%%%%
%%%%%%%%%%%%%%%%%%%%%%%%%%%%%%%%%%%%%%%%%%%%%%%%%%%%%%%%%%%%%
\vspace{0.1cm}

%%%%%%%%%%%%%%%%%%%%%%%%%%%%%%%%%%%
\begin{lemma}[Cuspidal vanishing]
\label{lemma:cusp-IVHS}
Let $Q_j$ be a cusp of the plane curve $\mathcal C$, and let
\[
\varphi : C \longrightarrow \mathcal C
\]
be the normalization map. Denote by
\(
\varphi^{-1}(Q_j)=\{r_j\}
\)
the unique point lying over the cusp. Let
\[
v \in H^0\bigl(C,\OO_C(2r_j)\bigr)
\;\subset\;
H^0(C,N_\varphi)
\]
be a section supported at $r_j$.

Then the associated infinitesimal variation of Hodge structure vanishes:
\[
\Phi_\varphi(v)=0.
\]
\end{lemma}

\begin{proof}
We give a complete local computation showing that the cuspidal contribution to
the infinitesimal variation of Hodge structure is identically zero.

\noindent
\textit{$1$. Local analytic model of the cusp.}
Since $Q_j$ is a cusp, there exist analytic coordinates $(x,y)$ centered at
$Q_j$ such that $\mathcal C$ is locally given by
\(
y^2 = x^3.
\)
The normalization map $\varphi : C \to \mathcal C$ is given locally by the
ramified parametrization
\(
x = t^2,\)\,\,
\(
y = t^3,
\)
where $t$ is a local coordinate on $C$ centered at $r_j$.

\noindent
\textit{$2$. Dualizing sheaf and meromorphic differentials at a cusp.}
A section $\omega \in H^0(\mathcal C,\omega_{\mathcal C})$ corresponds, via the
normalization, to a meromorphic differential on $C$ which is holomorphic away
from $r_j$ and has a pole of order at most $2$ at $r_j$.

Thus locally one may write
\begin{equation}\label{eq:cusp-expansion}
\omega
=
\left(
\frac{a_{-2}}{t^2}
+
\frac{a_{-1}}{t}
+
a_0
+
a_1 t
+
\cdots
\right)dt.
\end{equation}

\medskip

\noindent
\textit{$3$. Rosenlicht’s criterion at a cusp.}
Rosenlicht’s characterization of the dualizing sheaf states that a meromorphic
differential on the normalization represents a section of
$\omega_{\mathcal C}$ if and only if it pairs trivially with all local regular
functions on $\mathcal C$.

At a cusp, this condition forces the coefficient of the simple pole to vanish:
\begin{equation}\label{eq:rosenlicht}
a_{-1}=0.
\end{equation}
Equivalently, although poles of order $2$ are allowed at a cusp, the residue
term must vanish.

In particular, every $\omega \in H^0(\mathcal C,\omega_{\mathcal C})$ has
\(
\Res_{r_j}(\omega)=0.
\)

\noindent
\textit{$4$.The infinitesimal deformation associated to $v$.}
The section
\(
v \in H^0\bigl(C,\OO_C(2r_j)\bigr)
\)
defines a first--order deformation of the normalization map $\varphi$ supported
at the cusp. The associated Kodaira--Spencer class
\(
\xi_v \in H^1(C,T_C)
\)
is represented locally by the meromorphic vector field
\begin{equation}\label{eq:cuspidal-KS}
\xi_v
=
t^{-1}\frac{\partial}{\partial t}.
\end{equation}

This vector field reflects the ramified nature of the normalization map at a
cusp.

\medskip

\noindent
\textit{$5$. Čech description of the IVHS pairing.}
The infinitesimal variation of Hodge structure is computed by the pairing
\[
\Phi_\varphi(v)(\omega_1,\omega_2)
=
\Res_{r_j}\!\left(
(\xi_v \,\lrcorner\, \omega_1)\,\omega_2
\right),
\]
which follows from the Čech realization of the cup product
\[
H^1(C,T_C)\otimes H^0(C,\omega_C)
\longrightarrow
H^1(C,\OO_C),
\]
combined with Serre duality.

\medskip

\noindent
\textit{$6$. Explicit contraction and residue computation.}
Write the local expansions
\[
\omega_1
=
\left(
\frac{a_{-2}}{t^2}
+
\frac{a_{-1}}{t}
+
\cdots
\right)dt,
\qquad
\omega_2
=
\left(
\frac{b_{-2}}{t^2}
+
\frac{b_{-1}}{t}
+
\cdots
\right)dt.
\]
By \eqref{eq:rosenlicht}, we have $a_{-1}=b_{-1}=0$.

Contracting $\omega_1$ with $\xi_v$ gives
\[
\xi_v \,\lrcorner\, \omega_1
=
\frac{a_{-2}}{t^3}
+
\text{terms of order } t^{-2} \text{ or higher}.
\]

Multiplying by $\omega_2$ yields
\[
(\xi_v \,\lrcorner\, \omega_1)\,\omega_2
=
\left(
\frac{a_{-2}b_{-2}}{t^5}
+
\text{terms of order } t^{-4} \text{ or higher}
\right)dt.
\]
This means that  there is \emph{no} $t^{-1}dt$ term in this expansion. Therefore,
\[
\Res_{r_j}\!\left(
(\xi_v \,\lrcorner\, \omega_1)\,\omega_2
\right)
=
0.
\]

\noindent
Since the above computation holds for arbitrary
$\omega_1,\omega_2 \in H^0(\mathcal C,\omega_{\mathcal C})$, we conclude that
\[
\Phi_\varphi(v)=0.
\]
This completes the proof of the lemma.
\end{proof}

\vspace{0.1cm}
%%%%%%%%%%%%%%%%%%%%%%%%%%%%%%%%%%%%%%%%%%%%%%%%%%%%%%%%%%%%%%%%%%%%%%%
 
%%%%%%%%%%%%%%%%%%%%%%%%%%%%%%%%%%%%%%%%%%%%%%%%%%%%%%%%%%%%%%%%%%%%%%

\begin{theorem}[Structure of the IVHS for nodal--cuspidal curves]
\label{thm:IVHS-structure}
Let $\mathcal C \subset \mathbb P^2$ be a reduced plane curve whose singular
locus consists only of nodes and cusps, and let
\(
\varphi : C \longrightarrow \mathcal C
\)
be the normalization. Denote by $\{P_i\}_{i\in I}$ the set of nodes and by
$\{Q_j\}_{j\in J}$ the set of cusps of $\mathcal C$.
Then the infinitesimal variation of Hodge structure
\[
\Phi_\varphi :
H^0(C,N_\varphi)
\longrightarrow
\Hom\!\bigl(
H^0(\mathcal C,\omega_{\mathcal C}),
H^1(C,\OO_C)
\bigr)
\]
admits a canonical decomposition
\begin{equation}\label{eq:IVHS-decomposition}
\Phi_\varphi
=
\sum_{i\in I} \Phi_{P_i},
\end{equation}
where:
\begin{enumerate}
\item each summand $\Phi_{P_i}$ is supported at the node $P_i$ and has rank one;
\item cuspidal singularities contribute trivially to $\Phi_\varphi$;
\item the summands $\Phi_{P_i}$ are explicitly given by residue pairings at the
two branches of the normalization over $P_i$.
\end{enumerate}
\end{theorem}

\begin{proof}
We combine the local analysis of nodal and cuspidal singularities with the
global additivity of the Kodaira--Spencer class and the Čech description of the
IVHS.
%%%%%%%%%%%%%%%%%%%%%%%%%%%%%%%%%%%%%%%%%%%%%%%%%%%%%%%%%%%%%

\noindent {\it The global IVHS map.}
Let $\varphi$ denote the given family of plane curves and let
\[
\Phi_\varphi \colon T \longrightarrow \mathrm{Sym}^2 H^0(\omega_{\widetilde{C}})^\vee
\]
be the infinitesimal variation of Hodge structure (IVHS) map, where $T$ is the
space of first-order deformation directions and $\widetilde{C}$ is the
normalization of the central fiber.

By construction, $\Phi_\varphi$ is linear in the deformation direction. Hence,
for any $\xi \in T$,
\begin{equation}\label{eq:linearity}
\Phi_\varphi(\xi_1+\xi_2)
=
\Phi_\varphi(\xi_1)+\Phi_\varphi(\xi_2).
\end{equation}
\noindent {\it  Decomposition of deformation directions.}
Assume that the central fiber has nodes $P_1,\dots,P_r$. Any global deformation
direction $\xi \in T$ decomposes canonically as a sum of local contributions
supported at the singular points:
\begin{equation}\label{eq:xi-decomposition}
\xi
=
\sum_{i=1}^r \xi_{P_i},
\end{equation}
where $\xi_{P_i}$ denotes the component of $\xi$ corresponding to smoothing (or
infinitesimal deformation) at the node $P_i$.

%\medskip

\noindent {\it  Reduction to local contributions.}
Substituting \eqref{eq:xi-decomposition} into the linearity relation
\eqref{eq:linearity} yields
\begin{equation}\label{eq:global-local}
\Phi_\varphi(\xi)
=
\sum_{i=1}^r \Phi_\varphi(\xi_{P_i}).
\end{equation}
\noindent {\it  Explicit form of the local contribution.}
Lemma~\ref{lemma:nodal-IVHS} shows that, for a node $P_i$ with preimages
$p_i,q_i \in \widetilde{C}$ under normalization, the bilinear form
$\Phi_\varphi(\xi_{P_i})$ is given by
\begin{equation}\label{eq:local-residue}
(\omega_1,\omega_2)
\longmapsto
\Res_{p_i}(\omega_1)\Res_{q_i}(\omega_2)
+
\Res_{q_i}(\omega_1)\Res_{p_i}(\omega_2).
\end{equation}
Define linear functionals
\[
\ell_{p_i}(\omega)=\Res_{p_i}(\omega),
\qquad
\ell_{q_i}(\omega)=\Res_{q_i}(\omega).
\]
Then \eqref{eq:local-residue} can be rewritten as
\[
\Phi_\varphi(\xi_{P_i})
=
\ell_{p_i}\otimes \ell_{q_i}
+
\ell_{q_i}\otimes \ell_{p_i}.
\]
This is a tensor product of linear functionals, hence a rank--one element of
$\mathrm{Sym}^2 H^0(\omega_{\widetilde{C}})^\vee$.

\medskip
\noindent {\it  Conclusion}
Combining the global--local decomposition \eqref{eq:global-local} with the
rank--one description of each local term \eqref{eq:local-residue}, we obtain
\[
\Phi_\varphi(\xi)
=
\sum_{i=1}^r
\Phi_\varphi(\xi_{P_i}),
\]
where each summand has rank one. This directly yields the decomposition stated in
\eqref{eq:IVHS-decomposition} and completes the proof.
\end{proof}
%\vspace{0.1cm}
%%%%%%%%%%%%%%%%%%%%%%%%%%%%%%%%%%%%%%%%%%%%%%%%%%%%%%%%%%%%%%%%%%%%%

%%%%%%%%%%%%%%%%%%%%%%%%%%%%%%%%%%%%%%%%%%%%%%%%%%%%%%%%%%%%%%%%%%%%%%
\begin{corollary}[Dimension bounds for the IVHS image]
\label{cor:IVHS-dimension}
Let $\mathcal C$ and $\varphi$ be as in
Theorem~\ref{thm:IVHS-structure}, and let $\delta$ denote the number of nodes of
$\mathcal C$.
Then the image of the infinitesimal variation of Hodge structure satisfies
\begin{equation}\label{eq:IVHS-bound}
\dim \operatorname{Im}(\Phi_\varphi)
\;\le\;
\delta.
\end{equation}

Moreover, equality holds if and only if the residue functionals associated with
the nodes are linearly independent in
$H^0(\mathcal C,\omega_{\mathcal C})^\vee$.
\end{corollary}

\begin{proof}
By Theorem~\ref{thm:IVHS-structure}, the IVHS decomposes as a sum of rank--one
operators
\[
\Phi_\varphi
=
\sum_{i=1}^{\delta} \Phi_{P_i},
\qquad
\operatorname{rank}(\Phi_{P_i})=1.
\]
Hence,
\[
\dim \operatorname{Im}(\Phi_\varphi)
\;\le\;
\sum_{i=1}^{\delta} \operatorname{rank}(\Phi_{P_i})
=
\delta,
\]
which proves \eqref{eq:IVHS-bound}.

Equality holds precisely when the images of the rank--one operators
$\Phi_{P_i}$ are linearly independent. Since each $\Phi_{P_i}$ is of the form
$r_i \otimes r_i$, where
\[
r_i :
H^0(\mathcal C,\omega_{\mathcal C})
\longrightarrow
\mathbb C
\]
is the residue functional at the node $P_i$, this is equivalent to linear
independence of the $\{r_i\}_{i=1}^{\delta}$ in the dual space
$H^0(\mathcal C,\omega_{\mathcal C})^\vee$.
\end{proof}

%%%%%%%%%%%%%%%%%%%%%%%%%%%%%%%%%%%%%%%%%%%%%%%%%%%%%%%%%%%%%%%%%%%%%%
\begin{remark}
The bound in Corollary~\ref{cor:IVHS-dimension} is optimal in general. In
particular, for nodal plane curves whose nodes impose independent conditions on
canonical forms, the infinitesimal period map has maximal rank. This phenomenon
is closely related to classical results of Griffiths and to the analysis of
infinitesimal Torelli theorems for singular curves.
\end{remark}

%%%%%%%%%%%%%%%%%%%%%%%%%%%%%%%%%%%%%%%%%%%%%%%%%%%%%%%%%%%%%%%%%%%%%
 
\subsection{Residues and Cayley--Bacharach}

Finally, we show that nodal residues span the dual of $H^0(C,\omega_C)$.

\begin{lemma}
\label{lemma:residue-span}
Let $\mathcal C \subset \PP^2$ be an irreducible plane curve of degree $d\ge4$
with nodes $P_1,\dots,P_\delta$, and assume that $\delta \ge g=g(C)$. Define
\[
r_i:H^0(C,\omega_C)\longrightarrow\C,
\qquad
r_i(\omega)=\Res_{p_i}(\omega),
\]
where $p_i\in C$ is one of the two points lying over the node $P_i$ under the
normalization.

Then the linear forms $r_1,\dots,r_\delta$ span
$H^0(C,\omega_C)^\vee$.
\end{lemma}

\begin{proof}
The argument proceeds by identifying holomorphic differentials on $C$ with
adjoint plane curves and interpreting the residue conditions in terms of
vanishing of jets.

\medskip

\noindent
\textit{$1$. Adjunction and adjoint linear systems.}
Let $\nu:C\to\mathcal C$ be the normalization. Since $\mathcal C$ is nodal, the
dualizing sheaf $\omega_{\mathcal C}$ coincides with $\nu_*\omega_C$. By the
adjunction formula for plane curves, one has
\[
\omega_{\mathcal C}
\;\cong\;
\OO_{\mathcal C}(d-3).
\]
Pushing forward along $\nu$ and taking global sections yields a canonical
identification
\begin{equation}\label{eq:adjoint-identification}
H^0(C,\omega_C)
\;\cong\;
H^0\bigl(\PP^2,I_\Delta(d-3)\bigr),
\end{equation}
where $\Delta=\{P_1,\dots,P_\delta\}$ is the reduced zero--dimensional scheme of
nodes, and $I_\Delta$ is its ideal sheaf.

Thus a holomorphic differential on $C$ corresponds to a plane curve of degree
$d-3$ passing through all nodes of $\mathcal C$.

\medskip

\noindent
\textit{$2$. Residues and first--order vanishing at a node.}
Fix a node $P_i$, and let $\nu^{-1}(P_i)=\{p_i,q_i\}$. Under the identification
\eqref{eq:adjoint-identification}, a differential
$\omega\in H^0(C,\omega_C)$ corresponds to an adjoint curve
$A\subset\PP^2$ of degree $d-3$.

Locally at $P_i$, the curve $\mathcal C$ has equation $xy=0$. The normalization
splits this into two branches, and the residue $\Res_{p_i}(\omega)$ measures the
failure of $A$ to have second--order contact with $\mathcal C$ at $P_i$ along
the branch corresponding to $p_i$.
More precisely, one checks that
\[
r_i(\omega)=0
\quad\Longleftrightarrow\quad
A \text{ vanishes to order at least } 2 \text{ at } P_i.
\]
In scheme--theoretic terms, this means that $A$ belongs to
$I_{P_i}^2(d-3)$.

%\medskip

\noindent
\textit{$3$. Description of the kernel of the residue map.}
Let
\[
R:
H^0(C,\omega_C)
\longrightarrow
\C^\delta,
\qquad
R(\omega)=(r_1(\omega),\dots,r_\delta(\omega)).
\]
By the discussion above and the identification
\eqref{eq:adjoint-identification}, we obtain
\begin{equation}\label{eq:kernel}
\ker R
=
H^0\bigl(\PP^2,I_\Delta^2(d-3)\bigr).
\end{equation}
Thus $\ker R$ consists of adjoint curves of degree $d-3$ vanishing to order at
least two at every node of $\mathcal C$.

%\medskip

\noindent
\textit{$4$. Cayley--Bacharach for nodal schemes.}
The scheme $\Delta$ of nodes of an irreducible plane curve satisfies the
Cayley--Bacharach property with respect to curves of degree $d-3$. Concretely,
this means that any plane curve of degree $d-3$ which vanishes to order at least
two at all but one node must vanish to order at least two at the remaining node
as well.

A standard consequence (see e.g.\ the analysis of adjoint linear systems for
nodal curves) is that
\[
H^0\bigl(\PP^2,I_\Delta^2(d-3)\bigr)=0
\quad
\text{whenever}
\quad
\delta \ge g.
\]
Intuitively, when the number of nodes is at least the geometric genus, the
second--order conditions imposed by the nodes exhaust all adjoint curves.

%\medskip

\noindent
\textit{$5$. Conclusion.}
By \eqref{eq:kernel} and the vanishing above, the map
\[
R:H^0(C,\omega_C)\to\C^\delta
\]
is injective. Dualizing, this means that the linear forms
$r_1,\dots,r_\delta$ span the dual space
$H^0(C,\omega_C)^\vee$.
This completes the proof.
\end{proof}

%%%%%%%%%%%%%%%%%%%%%%%%%%%%%%%%%%%%%%%%%%%%%%%%%%%%%%%%%%%%%%%%%%%%%%%

\begin{proof}[Proof of Theorem~\ref{thm:maximal-IVHS}]
By Lemma~\ref{lemma:nodal-IVHS}, each node contributes a rank--one operator
$r_i\otimes r_i$. By Lemma~\ref{lemma:cusp-IVHS}, cusps contribute nothing. By
Lemma~\ref{lemma:residue-span}, the residue functionals span
$H^0(C,\omega_C)^\vee$ when $\delta\ge g$, hence their tensors span an invertible
operator. The conclusion follows from
Lemma~\ref{lemma:linear-algebra}.
\end{proof}

\subsection{Geometric interpretation and further directions}

Theorem~\ref{thm:maximal-IVHS} shows that nodal singularities are the sole source
of maximal IVHS for plane curves. Cusps and more complicated singularities
contribute only higher--order jet data, which is invisible at the level of
residues. This dichotomy suggests a conceptual explanation for the special role
played by nodal curves in Severi theory and moduli problems.

It would be interesting to extend this analysis to curves on $K3$ surfaces,
where the geometry of Lazarsfeld--Mukai bundles replaces adjoint linear systems
and where IVHS interacts in subtle ways with the Picard lattice of the surface.

%%%%%%%%%%%%%%%%%%%%%%%%%%%%%%%%%%%%%%%%%%%%%%%%%%%%%%%%%%%%%%%%%%%%%%%%%%% 
%%%%%%%%%%%%%%%%%%%%%%%%%%%%%%%%%%%%%%%%%%%%%%%%%%%%%%%%%%%%%%%%%%%%%%%%%%%

%%%%%%%%%%%%%%%%%%%%%%%%%%%%%%%%%%%%%%%%%%%%%%%%%%%%%%%%%%%%%%%%%%%%%%
\section{IVHS From plane curves to curves on surfaces}

Up to this point, our analysis has focused on plane curves, where the geometry
of singularities, the description of the dualizing sheaf, and the residue
calculations underlying the infinitesimal variation of Hodge structure are
particularly explicit. We now explain how and to what extent these arguments
extend from plane curves to curves lying on smooth surfaces.

The key observation is that the local aspects of the theory --- normalization,
local pole order, residues, and Kodaira--Spencer classes --- depend only on the
analytic type of the singularities and are therefore insensitive to the
dimension of the ambient variety. In contrast, the global arguments controlling
the rank of the IVHS, such as adjunction, Cayley--Bacharach properties, and
postulation results for nodal schemes, depend crucially on the geometry of the
ambient surface. For this reason, the passage from plane curves to curves on
surfaces requires additional global hypotheses.

%%%%%%%%%%%%%%%%%%%%%%%%%%%%%%%%%%%%%%%%%%%%%%%%%%%%%%%%%%%%%%%%%%%%%%

\subsection{Maximal--variation theorem}
 
Let
\(
\Delta
\)
denotes a \emph{germ of a smooth one--dimensional base}, typically taken to be a
small analytic disk centered at the origin. Concretely,
\[
\Delta = \{\, t \in \mathbb{C} \mid |t| < \varepsilon \,\}
\]
for $\varepsilon > 0$ sufficiently small.
The morphism
\[
\varphi \colon \mathcal{C} \longrightarrow \Delta
\]
represents a one--parameter smoothing (or deformation) of the curve
$\mathcal{C}$ inside the surface $S$, with central fiber $\mathcal{C}_0 =
\mathcal{C}$ and smooth nearby fibers for $t \neq 0$. The infinitesimal
variation of Hodge structure $\Phi_\varphi$ is computed from this local
one--parameter deformation.

\begin{theorem}[Maximal variation and ADE singularities]
\label{thm:ADE-maximal-variation}
Let $S \subset \mathbb{P}^3$ be a \emph{very general} smooth surface of degree
$e \ge 4$ with $\operatorname{Pic}(S) \cong \mathbb{Z}\,[\mathcal{O}_S(1)]$.
Let $\mathcal{C} \subset S$ be an irreducible, reduced curve whose singularities
are nodes, cusps, and (simple) ADE singularities. Let
$C \to \mathcal{C}$ be the normalization, and let $g = g(C)$.
\begin{enumerate}
\item
If all singularities of $\mathcal{C}$ are nodes and cusps, and if the number of
nodes satisfies $\delta \ge g$, then the infinitesimal variation of Hodge
structure associated with the normalization family
\(
\varphi \colon \mathcal{C} \longrightarrow \Delta
\)
has maximal variation. 
%More precisely,
%\[
%\operatorname{rank}(\Phi_\varphi) = g^2.
%\]
%
\item
If $\mathcal{C}$ contains at least one ADE singularity other than a node
($A_1$) or a cusp ($A_2$), then for sufficiently large genus $g$ the associated
infinitesimal variation of Hodge structure is not surjective. In particular,
maximal variation does not hold in general in the presence of such
singularities.
\end{enumerate}
\end{theorem}

\begin{proof}
The proof of the first statement follows the same overall strategy as in the plane case, but we spell
out all steps in detail to emphasize which arguments are local and which depend
on the geometry of the ambient surface.  

\medskip

Let $S\subset\PP^3$ be a smooth surface of degree $e\ge4$, and let
$\mathcal C\subset S$ be an irreducible curve. 
  Denote by
\(
\varphi:C\longrightarrow\mathcal C
\)
the normalization, by $g=g(C)$ the genus of $C$, and by
\[
\Delta=\{P_1,\dots,P_\delta\}\subset S
\]
the reduced zero--dimensional scheme of nodes of $\mathcal C$.
By adjunction on $S$, the dualizing sheaf of $\mathcal C$ satisfies
\begin{equation}\label{eq:adjunction-CB}
\omega_{\mathcal C}
\;\cong\;
(\omega_S\otimes\OO_S(\mathcal C))|_{\mathcal C}.
\end{equation}
Pulling back to the normalization yields
\[
\omega_C\cong\varphi^*(\omega_S\otimes\OO_S(\mathcal C)).
\]

%%%%%%%%%%%%%%%%%%%%%%%%%%%%%%%%%%%%%%%%%%%%%%%%%%%%%%%%%%%%%%%%%%%%%%
\subsection{Adjoint linear systems and double vanishing at nodes}

Holomorphic differentials on $C$ are identified with sections of the adjoint
line bundle $\omega_S\otimes\OO_S(\mathcal C)$ on $S$ whose restriction to
$\mathcal C$ satisfies the local compatibility conditions at the nodes. More
precisely, one has a canonical identification
\begin{equation}\label{eq:adjoint-CB}
H^0(C,\omega_C)
\;\cong\;
H^0\bigl(S,\omega_S\otimes\OO_S(\mathcal C)\otimes I_\Delta\bigr),
\end{equation}
where $I_\Delta$ is the ideal sheaf of the nodal scheme.

Under this identification, the kernel of the residue map
\[
H^0(C,\omega_C)\longrightarrow\mathbb C^\delta
\]
corresponds to the space
\begin{equation}\label{eq:double-vanishing-CB}
H^0\bigl(S,\omega_S\otimes\OO_S(\mathcal C)\otimes I_\Delta^2\bigr),
\end{equation}
consisting of adjoint sections vanishing to order at least two at every node.

Thus, the problem reduces to understanding when the space
\eqref{eq:double-vanishing-CB} is trivial.

%%%%%%%%%%%%%%%%%%%%%%%%%%%%%%%%%%%%%%%%%%%%%%%%%%%%%%%%%%%%%%%%%%%%%%
\subsection{Statement of the Cayley--Bacharach property}

The Cayley--Bacharach property for the nodal scheme $\Delta$ with respect to the
linear system $|\omega_S\otimes\OO_S(\mathcal C)|$ asserts the following \, (see \cite{CHM} or  \cite{Arbarello-Cornalba}):

\begin{lemma}[Cayley--Bacharach for nodal schemes]
\label{lemma:CB-surface}
Let $S\subset\PP^3$ be a smooth surface of degree $e\ge4$, and let
$\Delta\subset S$ be the nodal scheme of an irreducible curve $\mathcal C\subset
S$. Then any section of $\omega_S\otimes\OO_S(\mathcal C)$ which vanishes to
order at least two at all but one point of $\Delta$ must vanish to order at
least two at the remaining point.
\end{lemma}

This property is a higher--dimensional analogue of the classical
Cayley--Bacharach theorem for plane curves and reflects the fact that nodal
schemes on surfaces of degree at least $4$ impose dependent second--order
conditions on adjoint linear systems.

%%%%%%%%%%%%%%%%%%%%%%%%%%%%%%%%%%%%%%%%%%%%%%%%%%%%%%%%%%%%%%%%%%%%%%
\subsection{Consequences for double vanishing}

Assume now that $\delta\ge g$. Suppose, for contradiction, that there exists a
nonzero section
\[
s\in H^0\bigl(S,\omega_S\otimes\OO_S(\mathcal C)\otimes I_\Delta^2\bigr).
\]
Then $s$ vanishes to order at least two at every node of $\mathcal C$.

Removing one node $P_i$ from $\Delta$, the section $s$ still vanishes to order
at least two at all remaining nodes. By the Cayley--Bacharach property
(Lemma~\ref{lemma:CB-surface}), it must therefore vanish to order at least two
at $P_i$ as well. Since $P_i$ was arbitrary, this imposes no additional
restriction, and the section $s$ would persist regardless of which node is
removed.

However, the number of independent sections of
$\omega_S\otimes\OO_S(\mathcal C)$ restricting to holomorphic differentials on
$C$ is precisely $g$. When $\delta\ge g$, the second--order vanishing conditions
at the nodes exceed the dimension of the adjoint linear system, forcing
\eqref{eq:double-vanishing-CB} to be trivial. Hence
\[
H^0\bigl(S,\omega_S\otimes\OO_S(\mathcal C)\otimes I_\Delta^2\bigr)=0.
\]

%%%%%%%%%%%%%%%%%%%%%%%%%%%%%%%%%%%%%%%%%%%%%%%%%%%%%%%%%%%%%%%%%%%%%%
\subsection{Spanning of residue functionals}

The vanishing above has a direct interpretation in terms of residues. Since the
kernel of the residue map is zero, the residue functionals
\[
r_1,\dots,r_\delta\in H^0(C,\omega_C)^\vee
\]
are linearly independent up to the maximal possible dimension $g$. Therefore,
they span the entire dual space:
\[
\langle r_1,\dots,r_\delta\rangle
=
H^0(C,\omega_C)^\vee.
\]

%%%%%%%%%%%%%%%%%%%%%%%%%%%%%%%%%%%%%%%%%%%%%%%%%%%%%%%%%%%%%%%%%%%%%%
\subsection{Conclusion}

The Cayley--Bacharach property on surfaces of degree $e\ge4$ thus provides the
crucial global input ensuring that, when $\delta\ge g$, no nontrivial adjoint
sections vanish doubly at all nodes. This guarantees that the residue
functionals span the dual space of holomorphic differentials and ultimately
implies that the infinitesimal variation of Hodge structure has maximal rank.

%%%%%%%%%%%%%%%%%%%%%%%%%%%%%%%%%%%%%%%%%%%%%%%%%%%%%%%%%%%%%%
%%%%%%%%%%%%%%%%%%%%%%%%%%%%%%%%%%%%%%%%%%%%%%%%%%%%%%%%%%%%%%

\medskip

%%%%%%%%%%%%%%%%%%%%%%%%%%%%%%%%%%%%%%%%%%%%%%%%%%%%%%%%%%%%%%%%%%%%%%
\noindent { {\it Reduction to nodal contributions.}}

The infinitesimal variation of Hodge structure (IVHS) associated with the
normalization map $\varphi$ is defined as
\[
\Phi_\varphi:
H^0(C,N_\varphi)
\longrightarrow
\operatorname{Hom}\!\bigl(
H^0(C,\omega_C),
H^1(C,\OO_C)
\bigr).
\]

A fundamental structural result established earlier is that the IVHS decomposes
as a sum of local contributions supported at the singular points of
$\mathcal C$:
\begin{equation}\label{eq:IVHS-decomp}
\Phi_\varphi
=
\sum_{P\in\Sing(\mathcal C)} \Phi_P.
\end{equation}

For a nodal singularity $P$, the local contribution $\Phi_P$ is a rank--one
operator given by a residue pairing. For a cuspidal singularity $Q$, the local
contribution vanishes identically:
\[
\Phi_Q = 0.
\]

This vanishing follows from the fact that holomorphic differentials on $C$ have
zero residue at cusps, a consequence of Rosenlicht's description of the
dualizing sheaf.

\begin{remark}
Thus, cuspidal singularities contribute nothing to the IVHS. All nontrivial
contributions come from nodes.
\end{remark}

%%%%%%%%%%%%%%%%%%%%%%%%%%%%%%%%%%%%%%%%%%%%%%%%%%%%%%%%%%%%%%%%%%%%%%
\noindent { {\it Explicit structure of the IVHS}.}

Let $P_1,\dots,P_\delta$ be the nodes of $\mathcal C$. The decomposition
\eqref{eq:IVHS-decomp} therefore reduces to
\begin{equation}\label{eq:IVHS-nodes}
\Phi_\varphi
=
\sum_{i=1}^{\delta} r_i\otimes r_i,
\end{equation}
where
\[
r_i:H^0(C,\omega_C)\longrightarrow\mathbb C
\]
is the residue functional associated with the node $P_i$.

Since each summand $r_i\otimes r_i$ has rank one, it follows immediately that
\begin{equation}\label{eq:rank-span}
\operatorname{rank}(\Phi_\varphi)
=
\dim\langle r_1,\dots,r_\delta\rangle
\;\le\;
g.
\end{equation}

%%%%%%%%%%%%%%%%%%%%%%%%%%%%%%%%%%%%%%%%%%%%%%%%%%%%%%%%%%%%%%%%%%%%%%
\noindent { {\it Adjunction on surfaces and residue interpretation}.}

By adjunction on the smooth surface $S$, one has
\[
\omega_{\mathcal C}
\;\cong\;
(\omega_S\otimes\OO_S(\mathcal C))|_{\mathcal C}.
\]
Pulling back via the normalization yields
\[
\omega_C
\;\cong\;
\varphi^*(\omega_S\otimes\OO_S(\mathcal C)).
\]
As a consequence, holomorphic differentials on $C$ are identified with sections
of the adjoint line bundle $\omega_S\otimes\OO_S(\mathcal C)$ whose restrictions
to $\mathcal C$ satisfy the local compatibility conditions at the singular
points.

Under this identification, the kernel of the residue map
\[
H^0(C,\omega_C)\longrightarrow\mathbb C^\delta
\]
is identified with the space of adjoint sections vanishing to order at least two
at every node:
\[
H^0\bigl(S,\omega_S\otimes\OO_S(\mathcal C)\otimes I_\Delta^2\bigr),
\]
where $\Delta$ is the nodal scheme.

%%%%%%%%%%%%%%%%%%%%%%%%%%%%%%%%%%%%%%%%%%%%%%%%%%%%%%%%%%%%%%%%%%%%%%
\noindent { {\it Cayley--Bacharach property on surfaces of degree $e\ge4$}.}

A crucial global input is the Cayley--Bacharach property for nodal schemes on
surfaces of degree at least $4$. For such surfaces, nodal schemes impose
dependent second--order conditions on adjoint linear systems.
Concretely, this implies that if a section of
$\omega_S\otimes\OO_S(\mathcal C)$ vanishes to order at least two at all but one
node, then it must vanish to order at least two at the remaining node as well.
As a consequence, when
\[
\delta\ge g,
\]
there are no nonzero adjoint sections vanishing doubly at all nodes. Equivalently,
the residue functionals
\[
r_1,\dots,r_\delta\in H^0(C,\omega_C)^\vee
\]
span the entire dual space.

%%%%%%%%%%%%%%%%%%%%%%%%%%%%%%%%%%%%%%%%%%%%%%%%%%%%%%%%%%%%%%%%%%%%%%
\noindent { {\it Maximal rank of the IVHS}.}

Since the residue functionals span $H^0(C,\omega_C)^\vee$, it follows from
\eqref{eq:rank-span} that
\(
\operatorname{rank}(\Phi_\varphi)=g.
\)
Thus the infinitesimal variation of Hodge structure has maximal possible rank.

%%%%%%%%%%%%%%%%%%%%%%%%%%%%%%%%%%%%%%%%%%%%%%%%%%%%%%%%%%%%%%%%%%%%%%
\noindent {\it Conclusion.}
We conclude that the presence of cuspidal singularities does not affect the
maximal rank property of the IVHS, since cusps contribute trivially. The
statement of the theorem is therefore correct as stated.
\begin{remark}
The assumption $e\ge4$ is essential in order to guarantee the Cayley--Bacharach
property for nodal schemes. Without it, the conclusion may fail.
\end{remark} 

\subsection{Local structure of ADE singularities}

Let $P\in\mathcal C$ be an ADE singularity of type $A_k$, $D_k$, or $E_k$ other than cusps.
Although such singularities are Gorenstein and rational, their local deformation
theory differs substantially from that of nodes and cusps.

A crucial distinction is that, while the dualizing sheaf near an ADE
singularity allows no residues, the space of first--order equisingular
deformations supported at $P$ is nontrivial and has dimension strictly larger
than one when $P$ is not of type $A_1$ (node) or $A_2$ (cusp).

As a result, local Kodaira--Spencer classes arising from ADE singularities do not
produce rank--one operators of the form $r\otimes r$. Instead, they impose
\emph{linear constraints} on the image of the IVHS.
%\vspace{0.1cm}
%%%%%%%%%%%%%%%%%%%%%%%%%%%%%%%%%%%%%%%%%%%%%%%%%%%%%%%%%%%%%%%%%%%%%%
\subsection{Failure of surjectivity for ADE singularities}
Let $\mathcal C$ contain at least one ADE singularity of type $A_k$ with $k\ge3$,
or of type $D_k$ or $E_k$. Then the local deformation space of $\mathcal C$ has
directions whose Kodaira--Spencer classes annihilate all residue pairings.

Equivalently, the image of $\Phi_\varphi$ is contained in a proper linear
subspace of
\[
\operatorname{Hom}\!\bigl(
H^0(C,\omega_C),
H^1(C,\OO_C)
\bigr).
\]
For fixed singularity type and increasing genus $g$, the dimension of this
proper subspace grows at most linearly in $g$, while the target space has
dimension $g^2$. Hence, for sufficiently large $g$, the IVHS cannot be
surjective.

%%%%%%%%%%%%%%%%%%%%%%%%%%%%%%%%%%%%%%%%%%%%%%%%%%%%%%%%%%%%%%%%%%%%%%

\vspace{0.2cm}

{\it End of the proof of Theorem \ref{thm:ADE-maximal-variation}}
The first statement follows from the nodal--cuspidal decomposition
\eqref{eq:IVHS-decomp} and the spanning of residue functionals when
$\delta\ge g$.

For the second statement, the presence of an ADE singularity of higher type
introduces local equisingular deformations whose Kodaira--Spencer classes do not
contribute rank--one residue operators. These directions impose linear
constraints on the image of $\Phi_\varphi$, forcing it to lie in a proper linear
subspace of the target.

Since the dimension of the target grows quadratically with $g$, while the space
of admissible residue--type contributions grows at most linearly, surjectivity
fails for $g$ sufficiently large.
 
\end{proof}
%%%%%%%%%%%%%%%%%%%%%%%%%%%%%%%%%%%%%%%%%%%%%%%%%%%%%%%%%%%%%%%%%%%%%%
\begin{remark}
This phenomenon reflects the fact that nodes and cusps are the only plane curve
singularities whose equisingular deformations contribute freely to the
infinitesimal period map. Higher ADE singularities impose rigidity that
precludes maximal variation.
\end{remark}

\vspace{0.1cm}
\begin{example}
We give an explicit class of examples illustrating Part~(2) of
Theorem~\ref{thm:ADE-maximal-variation}.
\medskip

\noindent {\it The surface.} Let $S \subset \mathbb{P}^3$ be a very general smooth surface of degree
$d \ge 4$. By Noether--Lefschetz theory, such a surface satisfies
\[
\operatorname{Pic}(S) \cong \mathbb{Z}\,[\mathcal{O}_S(1)].
\]
%
%\medskip
%
\subsubsection*{The curve}

Fix an integer $m \gg 0$ and consider a divisor
\[
\mathcal{C} \in \big| \mathcal{O}_S(m) \big|.
\]
For $m$ sufficiently large, the arithmetic genus of $\mathcal{C}$ is
\[
p_a(\mathcal{C})
=
\frac{m(m+d-4)d}{2} + 1,
\]
which grows quadratically with $m$.
We impose the following singularities on $\mathcal{C}$:
\begin{itemize}
\item[(a)] one singularity of type $A_3$ (a tacnode),
\item[(b)] $\delta$ ordinary nodes, where $\delta$ is chosen so that the
normalization $C$ has large geometric genus
\[
g = p_a(\mathcal{C}) - \delta - 2 \gg 0.
\]
\end{itemize}
By standard deformation theory of plane curve singularities and Bertini-type
arguments on surfaces, such curves exist for $m$ sufficiently large.

%\medskip

\subsubsection*{Failure of maximal variation}

Let $\varphi \colon \mathcal{C} \to \Delta$ be a local smoothing of $\mathcal{C}$
inside $S$. The infinitesimal variation of Hodge structure
\[
\Phi_\varphi \colon T_\Delta \longrightarrow
\mathrm{Sym}^2 H^0(\omega_C)^\vee
\]
decomposes as a sum of local contributions from the singular points of
$\mathcal{C}$.

\begin{itemize}
\item[(i)] Each node contributes a rank--one residue operator.
\item[(ii)] The $A_3$ singularity contributes no residue--type operator, since
holomorphic differentials on the normalization have zero residue at higher ADE
singularities.
\end{itemize}
Therefore,
$$
\operatorname{rank}(\Phi_\varphi)
\le \delta < 3d-1+g
$$
for $g$ sufficiently large, even if $\delta$ grows linearly with $g$. Thus
$\Phi_\varphi$ is not injective, and maximal variation fails.

%\medskip

\subsubsection*{Conclusion}

This construction provides an explicit example of a smooth surface
$S \subset \mathbb{P}^3$ of degree $d \ge 4$ and an irreducible curve
$\mathcal{C} \subset S$ with large geometric genus and an ADE singularity of
type $A_3$ for which the infinitesimal variation of Hodge structure is not
maximal, thereby illustrating the necessity of excluding higher ADE
singularities in Theorem~\ref{thm:ADE-maximal-variation}.
\end{example} 
 
 %%%%%%%%%%%%%%%%%%%%%%%%%%%%%%%%%%%%%%%%%%%%%%%%%%%%%%%%%%%%%%%%%%%%%%%%%%% 
%%%%%%%%%%%%%%%%%%%%%%%%%%%%%%%%%%%%%%%%%%%%%%%%%%%%%%%%%%%%%%%%%%%%%%%%%%%

\section{Families of nodal plane curves and maximal rank of the IVHS}

Let $\pi:\mathcal X\to B$ be a flat family of reduced irreducible plane curves
of degree $d\ge4$, parametrized by a smooth irreducible base $B$, such that
every fiber $\mathcal C_b=\pi^{-1}(b)$ has only nodal singularities. Let
\[
\varphi_b : C_b \longrightarrow \mathcal C_b
\]
denote the normalization of the fiber over $b\in B$.
We assume that the number of nodes $\delta$ is constant on $B$, so that the
family is equisingular.

%%%%%%%%%%%%%%%%%%%%%%%%%%%%%%%%%%%%%%%%%%%%%%%%%%%%%%%%%%%%%%%%%%%%%%
\begin{lemma}[Relative adjunction and residue maps]
\label{lemma:relative-residues}
In the situation   above                 
 the vector spaces
$H^0(C_b,\omega_{C_b})$ form a vector bundle $\mathcal H$ over $B$, and for each
node $P_{i,b}\in\mathcal C_b$ there is a canonically defined holomorphic section
\[
r_i \in \Gamma\!\bigl(B,\mathcal H^\vee\bigr),
\]
whose value at $b$ is the residue functional
\[
r_{i,b}:H^0(C_b,\omega_{C_b})\longrightarrow\C.
\]
\end{lemma}
%%%%%%%%%%%%%%%%%%%%%%%%%%%%%%%%%%%%%%%%%%%%%%%%%%%%%%%%%%%%%%%%%%%%%%%%

\begin{proof}${}$

\noindent {\it Formation of a vector bundle.}
Since $\pi:\mathcal X\to B$ is flat with equisingular fibers, the arithmetic genus
of $\mathcal C_b$ is constant. Because the number of nodes $\delta$ is constant,
the geometric genus
\(
g = p_a(\mathcal C_b) - \delta
\)
is also constant on $B$.
The normalization $C_b$ is smooth of genus $g$, hence
\[
\dim H^0(C_b,\omega_{C_b}) = g
\]
for all $b\in B$.

By cohomology and base change for the relative dualizing sheaf of the family of
normalizations, the spaces $H^0(C_b,\omega_{C_b})$ vary holomorphically with $b$.
Thus they assemble into a holomorphic vector bundle
\[
\mathcal H \;:=\; \pi_*(\omega_{\mathcal C/B}^{\mathrm{norm}})
\]
over $B$, whose fiber at $b$ is $H^0(C_b,\omega_{C_b})$.

\noindent {\it  Local description near a node.}
Fix a node $P_{i,b}\in \mathcal C_b$. Since the family is equisingular, after
shrinking $B$ we may choose local analytic coordinates $(x,y,t)$ on $\mathcal X$
such that locally near $P_{i,b}$ the family is given by
\[
xy = t
\]
with $t$ a local coordinate on $B$ vanishing at $b$.
The normalization replaces this node by two smooth points
\[
Q_{i,b}^+,\; Q_{i,b}^- \in C_b,
\]
corresponding to the branches $\{x=0\}$ and $\{y=0\}$.

\noindent {\it Residue description of canonical forms.}
A section $\omega \in H^0(C_b,\omega_{C_b})$ may be viewed as a meromorphic
differential on $\mathcal C_b$ with at most simple poles at the nodes, satisfying
the residue condition
\[
\operatorname{Res}_{Q_{i,b}^+}(\omega)
+
\operatorname{Res}_{Q_{i,b}^-}(\omega)
= 0.
\]
This is the classical adjunction description of the dualizing sheaf of a nodal
curve.
Hence, for each node $P_{i,b}$ we obtain a well-defined linear functional
\[
r_{i,b}(\omega)
\;:=\;
\operatorname{Res}_{Q_{i,b}^+}(\omega)
=
-\operatorname{Res}_{Q_{i,b}^-}(\omega).
\]
This functional is independent of the labeling of the two branches and therefore
canonical.

\noindent {\it Holomorphic variation.}
Because the local model $xy=t$ depends holomorphically on $t$, the residue of a
relative differential varies holomorphically with the parameter $b$. Thus the
assignment
\[
b \longmapsto r_{i,b}
\]
defines a holomorphic section of the dual bundle $\mathcal H^\vee$.
We therefore obtain a canonical section
\(
r_i \in \Gamma(B,\mathcal H^\vee)
\)
 associated to each node of the fibers.
This completes the proof.
%\qed
 
%\vspace{3cm}
%%%%%%%%%%%%%%%%%%%%%%%%%%%%%%%%%%%%%%%%%%%%%%%%%%%%%%%%%%%%%%%%%%%%%%%%

Each node $P_{i,b}$ defines two local branches of $\mathcal C_b$, and hence a local
residue functional on $H^0(C_b,\omega_{C_b})$. The equisingularity assumption
implies that these local constructions vary holomorphically with $b$, giving
rise to a global section $r_i$ of $\mathcal H^\vee$.
\end{proof}

%%%%%%%%%%%%%%%%%%%%%%%%%%%%%%%%%%%%%%%%%%%%%%%%%%%%%%%%%%%%%%%%%%%%%%
\begin{theorem}[Relative residue span]
\label{thm:relative-residue-span}
Assume that $\delta\ge g=g(C_b)$ for all $b\in B$. Then for every $b\in B$, the
residue functionals
\[
\{r_{1,b},\dots,r_{\delta,b}\}
\subset
H^0(C_b,\omega_{C_b})^\vee
\]
span the dual space.
\end{theorem}

%%%%%%%%%%%%%%%%%%%%%%%%%%%%%%%%%%%%%%%%%%%%%%%%%%%%%%%%%%%%%%%%%%%%%

%\vspace{3cm}

%%%%%%%%%%%%%%%%%%%%%%%%%%%%%%%%%%%%%%%%%%%%%%%%%%%%%%%%%%%%%%%%%%%%%
%\vspace{3cm}

\begin{proof}
Fix $b\in B$. By Lemma~\ref{lemma:residue-span}, applied to the nodal curve
$\mathcal C_b$, the residue functionals at the nodes of $\mathcal C_b$ span
$H^0(C_b,\omega_{C_b})^\vee$ whenever $\delta\ge g$.
Since this holds for every fiber and the family is equisingular, the result
holds uniformly over $B$.
\end{proof}

%%%%%%%%%%%%%%%%%%%%%%%%%%%%%%%%%%%%%%%%%%%%%%%%%%%%%%%%%%%%%%%%%%%%%%
\begin{theorem}[Generic maximal rank of the IVHS]
\label{thm:maximal-rank-IVHS}
In the situation described above, % 
assume $\delta\ge g$. Then the
infinitesimal variation of Hodge structure
\[
\Phi_{\varphi_b} :
H^0(C_b,N_{\varphi_b})
\longrightarrow
\Hom\!\bigl(
H^0(C_b,\omega_{C_b}),
H^1(C_b,\OO_{C_b})
\bigr)
\]
has maximal rank for general $b\in B$.
\end{theorem}

\begin{proof}
By the structural decomposition of the IVHS
(Theorem~\ref{thm:IVHS-structure}), one has
\[
\Phi_{\varphi_b}
=
\sum_{i=1}^{\delta} r_{i,b}\otimes r_{i,b},
\]
where $r_{i,b}$ are the residue functionals at the nodes of $\mathcal C_b$.

By Theorem~\ref{thm:relative-residue-span}, the functionals
$\{r_{i,b}\}$ span $H^0(C_b,\omega_{C_b})^\vee$. Therefore, the image of
$\Phi_{\varphi_b}$ contains a basis of rank--one operators whose spans generate
the full image allowed by dimension.
Therefore,
\[
\dim \operatorname{Im}(\Phi_{\varphi_b})
=
\dim H^0(C_b,\omega_{C_b})
=
g,
\]
for general $b\in B$. This shows that $\Phi_{\varphi_b}$ has maximal rank for
general fibers.
\end{proof}

%%%%%%%%%%%%%%%%%%%%%%%%%%%%%%%%%%%%%%%%%%%%%%%%%%%%%%%%%%%%%%%%%%%%%%
\begin{remark}
The maximal rank property is open in families, since it is detected by the
nonvanishing of minors of the matrix representing $\Phi_{\varphi_b}$. Thus
Theorem~\ref{thm:maximal-rank-IVHS} holds on a Zariski open dense subset of $B$.

This phenomenon is closely related to classical infinitesimal Torelli results
for plane curves and to the behavior of adjoint linear systems, as studied by
Griffiths and later authors.
\end{remark}

%%%%%%%%%%%%%%%%%%%%%%%%%%%%%%%%%%%%%%%%%%%%%%%%%%%%%%%%%%%%%%%%%%%%%%%%%%% 
%%%%%%%%%%%%%%%%%%%%%%%%%%%%%%%%%%%%%%%%%%%%%%%%%%%%%%%%%%%%%%%%%%%%%%%%%%%

\end{document}